\documentclass{article}
\usepackage[utf8]{inputenc}

\usepackage{amsmath}
\usepackage{mathtools}
\usepackage{amsfonts}
\usepackage{amssymb,tikz}
\usepackage{latexsym}
\usepackage{bbm}
\usepackage{arydshln}
\usepackage{algorithm}
\usepackage{algpseudocode,float}
\numberwithin{algorithm}{section}
\usepackage{multicol}
\usepackage{enumitem}
\usepackage[format=plain,labelfont=bf]{caption}
\usepackage{graphicx}
\usepackage{subcaption}
\usepackage{placeins} 
\usepackage{tablefootnote}
\usepackage{amsthm}

\usepackage{authblk}  % For author affiliations
\usepackage{fancyhdr} % Optional, for custom headers/footers
\usepackage[colorlinks=true, linkcolor=blue, citecolor=blue]{hyperref}
\usepackage[nameinlink,capitalize]{cleveref}
\crefname{equation}{}{}

\def\R{{\mathbb R}}
\def\C{{\mathbb C}}
\def\N{{\mathbb N}}

\def\S{{\mathbb S}}
\def\Id{\boldsymbol{I}}
\def\ideal{{\mathcal I}}

\newcommand{\norm}[1]{\left\|#1\right\|}
\newcommand{\mat}[1]{\boldsymbol{#1}}
\renewcommand{\vec}[1]{\boldsymbol{#1}}

\newtheorem{lemma}{Lemma}[section]
\newtheorem{definition}{Definition}[section]
\newtheorem{corollary}{Corollary}[section]
\newtheorem{theorem}{Theorem}[section]
\newtheorem{proposition}{Proposition}[section]
\newtheorem{remark}{Remark}[section]
\newtheorem{example}{Example}[section]

\usepackage{geometry}
\title{Spectral Methods for Polynomial Optimization}

\author{
  Elvira Moreno\textsuperscript{\textdagger} and 
  Venkat Chandrasekaran\textsuperscript{\textdagger,\textdaggerdbl}
}

\date{
  \textsuperscript{\textdagger}Department of Computing and Mathematical Sciences \\
  \textsuperscript{\textdaggerdbl}Department of Electrical Engineering \\
  California Institute of Technology \\
  Pasadena, CA 91125 \\
  \vspace{1ex}
  July 2025
}

\begin{document}

\maketitle

\begin{abstract} We present a hierarchy of tractable relaxations to obtain lower bounds on the minimum value of a polynomial over a constraint set defined by polynomial equations. In contrast to previous convex relaxation techniques for this problem, our method is based on computing the smallest generalized eigenvalue of a pair of matrices derived from the problem data, which can be accomplished for large problem instances using off-the-shelf software. We characterize the algebraic structure in a problem that facilitates the application of our framework, and we observe that our method is applicable for all polynomial optimization problems with bounded constraint sets. Our construction also yields a nested sequence of structured convex outer approximations of a bounded algebraic variety with the property that linear optimization over each approximation reduces to an eigenvalue computation. Finally, we present numerical experiments on representative problems in which we demonstrate the scalability of our approach compared to convex relaxation methods derived from sums-of-squares certificates of nonnegativity.

%Our construction also admits an interpretation in which we obtain a sequence of structured convex outer approximations of a bounded algebraic variety with the property that linear optimization over these approximations reduces to an eigenvalue computation.

\end{abstract}

% REQUIRED

%%%%%%%   Sections 

\section{Motivation}\label{sec: Introduction}

%Given real $n$-variate polynomials $p,g_1,\dots,g_{\ell}$ we consider the problem of minimizing $p$ over the common zero set of $g_1,\dots,g_{\ell}$:

We consider the problem of minimizing a real polynomial $p$ in $n$ variables over the common zero set of a finite collection of real $n$-variate polynomials $g_1,\dots,g_{\ell}$:
\begin{equation}\label{POP_init}
\nu\coloneq\min_{\vec{x}\in\R^n} \;\; p(\vec{x}) \;\; \mbox{s.t.} \;\;  g_1(\vec{x})=\dots=g_{\ell}(\vec{x})=0.
\end{equation}
The formulation \cref{POP_init} describes a general polynomial optimization problem (POP), as it accommodates polynomial inequality constraints of the form $g(\vec{x})\geq 0$ by introducing an extra variable $y$ and enforcing $g(\vec{x})=y^2$. POPs include linear programs, quadratic programs, and combinatorial optimization problems as special cases, and they have numerous applications~\cite{Ahmadi2015applications, hall2019engineeringbusinessapplicationssum, Lasserre2015, moniqueAppl, Parrilo2012Applic}. 
%applications in numerous areas including control theory, machine learning, energy systems and power flow, structural engineering, finance, and signal processing, among others.

As the POP formulation \cref{POP_init} encompasses NP-hard problems such as maximum-cut~\cite{Karp1972}, we should not expect to solve it in a computationally efficient manner in general. This has motivated the development of tractable relaxations to bound the optimal value $\nu$. Perhaps the most prominent approach to deriving such relaxations is the moment/sums-of-squares (SOS) hierarchy~\cite{Lasserre2001, ParriloTesis, Parrilo2003}, which yields a sequence of semidefinite programming (SDP) problems whose optimal values converge to $\nu$ from below under suitable assumptions. While this method has strong theoretical guarantees and is computationally tractable in principle, its applicability is limited to modest-sized problems due to the well-documented scalability challenges associated with solving large SDPs~\cite{sdpScalable,WangRegularization}. This limitation motivates the development of alternative tractable and scalable relaxations for POPs. %yield increasingly tighter lower bounds on $\nu$. %provide increasingly tighter lower bounds that converge asymptotically to $\nu$ under suitable assumptions %While computationally tractable and theoretically well-grounded, such SDP-based methods suffer from poor scalability, rendering them impractical for large-scale problems \TODO{cite}. 

\subsection{Contributions} \label{sec: Outline}

In this paper, we establish lower bounds on \cref{POP_init} by leveraging eigenvalue problems as a computational primitive. The \emph{generalized eigenvalue problem} (GEP) for matrices $\mat{A}$ and $\mat{B}$ is to identify $\lambda \in \C$ and $\vec{v} \in \C^d$ such that:
\begin{equation}\label{GEP}
\mat{A}\vec{v}=\lambda \mat{B}\vec{v},\; \vec{v}\neq 0. \end{equation}
The solutions $\lambda\in \C$ and $\vec{v}\in\C^d$ to \cref{GEP} are termed \emph{generalized eigenvalues} and \emph{eigenvectors}, respectively.  Generalized eigenvalues for \cref{GEP} need not exist, and when they exist, they need not be real or even finite. Nonetheless, when $\mat{A}$ is symmetric and $\mat{B}$ positive definite, all solutions $\lambda$ to \cref{GEP} are real and can be efficiently computed using reliable off-the-shelf solvers~\cite{Anderson1999, Lehoucq1998}. Moreover, under these conditions, the minimum generalized eigenvalue of $\mat{A}$ and $\mat{B}$ admits the following SDP interpretation:
    \begin{equation} \label{OP: SR}
        \lambda_{\min}(\mat{A},\mat{B})= \max_{\gamma\in\R} \;\; \gamma \; \; \mbox{s.t.}\;\; \mat{A}-\gamma \mat{B} \succeq 0.
    \end{equation}  
Our approach is to design problems of the form \cref{OP: SR} with $\mat{A}$ symmetric and $\mat{B}\succ0$, such that the associated optimal value is a lower bound on the optimal value $\nu$ of the POP~\cref{POP_init}. As 
\cref{OP: SR} can be reliably computed using eigenvalue solvers for much larger problem sizes than is feasible with general-purpose SDP solvers, this approach yields bounds on POPs that are beyond the reach of the moment–SOS hierarchy (see \cref{sec: Experiments}). In keeping with their connection to eigenvalue problems, we refer to the formulations \cref{OP: SR} as \emph{spectral relaxations} and to the bounds that they produce as \emph{spectral bounds}.
%The motivation for this approach is that \cref{OP: SR} can be reliably computed using off-the-shelf eigenvalue solvers for much larger problem sizes than are feasible with general-purpose SDP solvers.  Indeed, we demonstrate that our framework yields bounds on POPs that are beyond the reach of the moment–SOS hierarchy. 

In \cref{sec: Deriving Spectral Relaxations}, we begin by describing algebraic conditions on the constraint set and the objective polynomial of the POP~\cref{POP_init} that facilitate a natural approach for identifying matrices $\mat{A}$ and $\mat{B}\succ 0$ such that $\lambda_{\min}(\mat{A},\mat{B})\leq \nu$. Moreover, we present a hierarchy of such spectral relaxations, given by a sequence of matrices $\mat{A_k}$ and $\mat{B_k}\succ0$, typically of growing size, such that the corresponding minimum generalized eigenvalues $\lambda_{\min}(\mat{A_k},\mat{B_k})$ yield a non-decreasing sequence of lower bounds on $\nu$.

% is to identify algebraic structure in the problem that enables their formulation. We address this in \cref{sec: Identifying SRs}. \cref{sec: Initial SR} then presents a systematic approach for obtaining matrices $\mat{A}$ and $\mat{B}\succ 0$ such that $\lambda_{\min}(\mat{A},\mat{B})\leq \nu$. In \cref{sec: Hierarchy}, we derive a hierarchy of spectral relaxations for \cref{POP_init}; that is, we construct sequences of matrices $\mat{A_k}$ and $\mat{B_k}\succ0$, typically of increasing size, whose minimum generalized eigenvalues $\lambda_{\min}(\mat{A_k},\mat{B_k})$ yield a non-decreasing sequence of lower bounds on $\nu$. %This hierarchy can be viewed as nested within a hierarchy of SOS-based SDP relaxations, where at each level there exists an SDP relaxation whose feasible set contains that of the corresponding spectral relaxation.

In \cref{sec: SR amenable ideals}, we build on the analysis in \cref{sec: Deriving Spectral Relaxations} to characterize the constraint sets of \cref{POP_init} for which our spectral relaxations can be obtained for any objective polynomial $p$. We present problem families arising in various application domains that satisfy these criteria and are well-suited to our framework. In particular, we discuss how any POP~\cref{POP_init} with bounded constraint set is amenable to our method.

%In \cref{sec: SR amenable ideals}, we build on the analysis in \cref{sec: Deriving Spectral Relaxations} to describe conditions on a constraint set of \cref{POP_init} under which our spectral relaxations can be obtained for any objective polynomial.  We present several problems families arising in various problem domains that satisfy these conditions, and to which our framework is readily applicable.  In particular, we also discuss how any POP~\cref{POP_init}  with a bounded constraint set is amenable to our method.

%In \cref{sec: Dual Perspective} we study spectral relaxations from a dual perspective. In more detail, the POP~\cref{POP_init} may be reformulated as the minimization of a linear functional (given by the coefficients of the objective polynomial $p$ in a suitable basis) over a polynomial image of the constraint set; \EM{Please see comment.} from this viewpoint, the duals of our spectral relaxations yield structured convex outer approximations to the constraint set. These approximations are given by linear images of a base of the cone of positive-semidefinite matrices, which we call \emph{spectratopes}. Spectratopes are a particular class of projected spectrahedra over which linear optimization reduces to a generalized eigenvalue computation. This dual viewpoint also yields spectratope outer approximations to the convex hull of an algebraic variety (i.e., the solution set of a system of polynomial equations), which we compare with previous outer approximations obtained from SOS-methods~\cite{thetabodies}.

In \cref{sec: Dual Perspective} we study spectral relaxations from a dual perspective.  Duals of the problem \cref{OP: SR} are given by convex programs in which linear functionals are optimized over projected spectrahedra that have a particular structure.  Specifically, these constraint sets are described as linear images of a base of the cone of positive-semidefinite matrices, which we call \emph{spectratopes}.  This dual viewpoint also yields spectratope outer approximations to the convex hull of an algebraic variety (i.e., the solution set of a system of polynomial equations), which we compare with previous outer approximations obtained from SOS-methods~\cite{thetabodies}.

% The dual of our spectral relaxations minimize the same linear functional but over a structured outer approximation 

% From a dual perspective, spectral relaxations for the POP~\cref{POP_init} arise from minimizing a linear functional -- defined by its objective polynomial $p$-- over structured outer approximations of (a polynomial image of) its constraint set. These approximations correspond to linear images of the set of trace-one positive semidefinite matrices, which we term \emph{spectratopes}, and have the property that linear optimization over them reduces to an eigenvalue computation. We elaborate on this perspective in \cref{sec: Derive Moment Relaxation}. In \cref{sec: Lambda Bodies}, we derive nested sequences of spectratope outer approximations to constraint sets that are amenable to spectral relaxations (as characterized in \cref{sec: Characterization}), and compare them with previous outer approximations obtained from SOS-methods~\cite{thetabodies}. % {\color{red} For illustration, we derive spectratope relaxations for two well-known combinatorial sets: the cut polytope and the stable set polytope.}

In \cref{sec: Experiments}, we present numerical experiments highlighting the performance of our framework on three problems -- computing the maximum-cut value of a graph, estimating the distance to a variety, and computing the spectral norm of a tensor.  While SOS methods yield tighter bounds, our results show that spectral relaxations provide a compelling alternative under memory and time constraints, as they scale to much larger problem instances and produce good-quality bounds significantly faster. 

Finally, we note that our development relies on some basic concepts from real algebraic geometry, and we provide the relevant background in \cref{sec: Background}.

\subsection{Related Work}\label{sec: Related Work} 

Numerous previous efforts have addressed the scalability challenges associated with SDP relaxations for POPs. \cref{sec: SDP scalability} provides a general overview. \cref{sec: spectral hierarchies} focuses on methods that are based on eigencomputations.

\subsubsection{Addressing Scalability Limitations of SDP-Based Relaxations}\label{sec: SDP scalability} 

A substantial body of work focuses on addressing the scalability limitations of SDPs arising from relaxations of POPs~\cite{sdpScalable}. In particular, recent methods derive moment-SOS-based SDP relaxations, whose solution is based on leveraging eigenvalue computations (in the spirit of spectral bundle methods \cite{Helmberg2000}), rather than interior-point methods~\cite{Mai2023}. Unlike our framework, these methods still aim to solve SDPs. A second line of work focuses on designing relaxations for POPs that replace SDPs derived from SOS certificates (see \cref{sec: Background}) with alternative optimization formulations.  The paper~\cite{Ahmadi2019} introduced the so-called diagonally-dominant sums-of-squares (DSOS) and scaled-diagonally-dominant sums-of-squares (SDSOS) hierarchies of linear programming (LP) and second-order cone programming (SOCP) relaxations for POPs.  Yet another class of methods derive relaxations based on relative entropy optimization and geometric programming \cite{Chandrasekaran2016, Dressler2017, Katthn2021, Murray2020, Murray2021}, with the underlying nonnegativity certificates being grounded in the arithmetic-geometric-mean inequality.   In contrast, our framework relies on a single eigenvalue computation to obtain a lower bound on the optimal value of the POP~\cref{POP_init}.  As eigenvalue solvers can scale to much larger problem sizes than convex optimization solvers, the methods we describe in this paper are especially relevant for POPs involving many variables.  Moreover, despite bypassing the need to use an optimization solver, our spectral relaxations produce bounds comparable to or better than those obtained from first-level DSOS and SDSOS relaxations for certain families of maximum-cut and tensor spectral norm problem instances (see Sections \ref{sec: Max-Cut} and  \ref{sec:TSN}).

%In particular, recent methods derive moment-SOS-based SDP relaxations for constrained POPs that can be reformulated as minimizing the largest eigenvalue of a matrix pencil, enabling the use of more efficient solution techniques in place of traditional interior-point methods~\cite{Mai2023}.

\subsubsection{Spectral Hierarchies for Polynomial Optimization} \label{sec: spectral hierarchies}

Most relevant to our development is recent work by Lovitz and Johnson~\cite{lovitz2024hierarchyeigencomputationspolynomialoptimization}, which presents a convergent hierarchy of spectral relaxations for minimizing a homogeneous polynomial over the Euclidean sphere, and more generally, for minimizing the inner product between a real tensor and a unit symmetric product tensor. Our hierarchy applies to a broader family of POPs, and specializes to the construction in~\cite{lovitz2024hierarchyeigencomputationspolynomialoptimization} for the case of homogeneous optimization over the sphere, recovering the same sequence of bounds for this class of problems (see \cref{sec: Hierarchy}).  Finally, a hierarchy of spectral upper bounds for polynomial minimization over certain structured families of sets (e.g., $\R^n$, $\R_+^n$, boxes, ellipsoids, or simplices) was proposed in~\cite{Lasserre2011}.  In contrast, our hierarchy yields lower bounds and can be applied to POPs with bounded constraint sets.

%Their key insight is to reduce the original POP to a complex Hermitian analogue, enabling the application of the Hermitian sum-of-squares (HSOS) hierarchy \TODO{cite} —a complex counterpart to the SOS hierarchy—  which, in this setting, simplifies to a sequence of eigenvalue problems.
%Closely related to the work of Lovitz et al, are convergent hierarchies of eigencomputations for constrained (complex) Hermitian optimization~\cite{DJL24,DAngelo2008}. In this work, we restrict our attention to real POPs.

%%%%%%%%%%%%%%%%%%%%%%%%%%%%%%%%%%%%%%%%%%%%%%%%%%%%%%%%%%%%%%%%%%%%%%%%%%%%%%%%%%%%%%%%%%%%%%%%%%%%%%%%%%%%%%%%%%%%%%%%%%%%%%%%%%%%%%%%%%%%%%%%%%%%%%%%%%%%%%%%%%%%%%%%%%%%%%%%%%%

\section{Background}\label{sec: Background}

We present here some key concepts from real algebraic geometry. We refer the reader to~\cite{idealsVarAlg} for a more extensive treatment of this topic.

The \emph{polynomial ring} $\R[\vec{x}]$ consists of all polynomials in the indeterminates $\vec{x}=(x_1,\dots,x_n)$ with real coefficients. We are concerned with minimizing a polynomial $p\in\R[\vec{x}]$ over the common set of real zeros of a finite list of polynomials $g_1, \dots, g_{\ell}\in \R[\vec{x}]$. In this context, it is useful to develop methods that depend solely on the constraint set, rather than on the specific polynomials $g_1, \dots, g_{\ell}$ used to define it. The notion of an ideal provides the framework for addressing this challenge.

A subset $\ideal\subset\R[\vec{x}]$ is called an \emph{ideal} in $\R[\vec{x}]$ if it closed under addition and satisfies the following absorption property: $g\in \ideal, ~ f\in \R[\vec{x}] ~ \Rightarrow ~ fg\in \ideal$. Given polynomials $g_1, \dots, g_{\ell}\in \R[\vec{x}]$, the set $\langle g_1, \dots, g_{\ell} \rangle\coloneqq\{f_1g_1+\dots+ f_{\ell}g_{\ell} : f_1,\dots,f_{\ell}\in\mathbb{R}[\vec{x}]\}$ of all `polynomial combinations' of $g_1,\dots,g_{\ell}$ is an ideal, commonly referred to as the \emph{ideal generated by} $g_1, \dots, g_{\ell}$. The real (affine) variety of an ideal $\ideal\subset\R[\vec{x}]$ is the collection of points in $\R^n$ at which all of the polynomials in $\ideal$ simultaneously vanish, that is, $V_\R(\ideal)\coloneq\{\vec{x}\in\R^n: g(\vec{x})=0, \forall g\in \ideal\}.$ One can check that $V_\R(\langle g_1, \dots, g_{\ell} \rangle)=\{\vec{x}\in\R^n: g_1(\vec{x})=\dots=g_{\ell}(\vec{x})=0\}$. Further, Hilbert's basis theorem states that any ideal in a polynomial ring over a field is finitely generated~\cite{idealsVarAlg}. Together, these insights imply that the zero set of any finite collection of polynomials corresponds to the variety of an ideal, and conversely, that the variety $V_\R(\ideal)$ of an ideal $\ideal\subset\R[\vec{x}]$ can always be described as the zero set of finitely many polynomials.  Hence, polynomial ideals represent the appropriate algebraic abstraction for working with constraint sets defined by finitely many polynomial equations.

Minimizing a polynomial $p$ over the variety $V_\R(\ideal)$ of an ideal $\ideal \subset \R[\vec{x}]$ is equivalent to maximizing over lower bounds for $p$ over $V_\R(\ideal)$. This task is difficult for general $p$ and $\ideal$. In fact, checking whether $p-\gamma$ is nonnegative over $V_\R(\ideal)$ for a fixed $\gamma\in\R$ is already NP-Hard in general~\cite{NpcompleteMurty}. With a view to developing tractable sufficient conditions for certifying nonnegativity, we begin with a basic observation. To certify that a polynomial $f\in\R[\vec{x}]$ is nonnegative over $V_\R(\ideal)$, it suffices to identify another polynomial $q\in\R[\vec{x}]$ that is evidently globally nonnegative and for which $f-q\in\ideal$. Motivated by this fact, it is convenient to reason about polynomial optimization over $V_\R(\ideal)$ in terms of the \emph{quotient ring} $\R[\vec{x}]/\ideal$, which consists of all equivalence classes of polynomials in $\R[\vec{x}]$ with respect to the following equivalence relation defined by $\ideal$ over $\R[\vec{x}]$: two polynomials $f,q\in\R[\vec{x}]$ are said to be \emph{equivalent} \emph{mod $\ideal$}, denoted $f\equiv_\ideal q$, if their difference $f-q$ is an element of $\ideal$. The equivalence class of a polynomial $f \in \R[\vec{x}]$ is denoted $f + \ideal$.

% With this in mind, it is useful to identify sufficient conditions for polynomial nonnegativity over $V_\R(\ideal)$ that define subsets of lower bounds for $p$ on $V_\R(\ideal)$ over which efficient optimization is possible. By maximizing over progressively larger subsets of certified lower bounds, one can achieve enhanced approximations for the optimal value of the original minimization problem. 
% With this in mind, it is useful to identify tractable means of certifying polynomial nonnegativity over $V_\R(\ideal)$. This allows one to search over subsets of lower bounds for $p$ on $V_\R(\ideal)$ over which efficient optimization is possible and obtain approximations to the optimal value of the associated POP. 

Of particular relevance to our discussion is the property that $\R[\vec{x}]/\ideal$ has vector space structure over $\R$, where addition and scalar multiplication are defined as $(f+\ideal)+(q+\ideal)=(f+q)+\ideal$ and $\eta(f+\ideal)=\eta f+\ideal$ for $\eta\in\R$, respectively. In addition, $\R[\vec{x}]/\ideal$ is equipped with a bilinear product: $(f+\ideal)(q+\ideal)=fq+\ideal$, which renders it a \emph{unital $\R$-algebra}. This structure allows us to reason about subalgebras of $\R[\vec{x}]/\ideal$ generated by the equivalence classes mod $\ideal$ of a finite set of polynomials. Given $f_1, \dots, f_m\in\R[\vec{x}]$, the corresponding subalgebra consists of all finite $\R$-linear combinations of finite products of the elements $f_1 + \ideal, \dots, f_m + \ideal$, and we denote it by $\mbox{Alg}_\ideal(f_1, \dots, f_m)$.

%Given polynomials $f_1, \dots, f_m \in \R[\vec{x}]$, we denote by $\mbox{Alg}_\ideal(f_1, \dots, f_m)$ the subalgebra of $\R[\vec{x}]/\ideal$ generated by their equivalence classes mod $\ideal$, that is, the set of all finite $\R$-linear combinations of finite products of the elements $f_1 + \ideal, \dots, f_m + \ideal$. 

As the dimension of $\R[\vec{x}]/\ideal$ (as a vector space) is often very large -- either infinite or finite but exponentially large in the problem dimension -- one is limited to searching for evidently nonnegative polynomials within fixed, modest-dimensional subspaces $U \subset \R[\vec{x}]/\ideal$. Given this consideration, a prominent approach to certifying nonnegative of a polynomial $f\in\R[\vec{x}]$ over $V_\R(\ideal)$ is based on SOS decompositions: a polynomial is SOS, and consequently globally nonnegative, if it can be expressed as $\sum_{i=1}^d{q_i^{2}(\vec{x})}$ for some polynomials $q_1, \dots, q_d\in\R[\vec{x}]$. The SOS approach proceeds in two steps. First, one identifies a vector of polynomials $\vec{z}=(z_1,\dots,z_d) \in \R[\vec{x}]^d$ such that $f+\ideal$ belongs to  the subspace $U$ of $\R[\vec{x}]/\ideal$ spanned by the products $z_i \cdot z_j + \ideal$ for $i,j = 1,\dots,d$. Then, one searches for a positive semidefinite \emph{Gram} matrix $\mat{M}\in\S^{d}$ for the polynomial $f$ with respect to the vector of quotient polynomials $\vec{z}+\ideal\in(\R[\vec{x}]/\ideal)^d$, that is, a matrix $\mat{M}\in\S^{d}_{+}$ such that $\vec{z}^T\mat{M} \vec{z}\equiv_\ideal f$. Given such an $\mat{M}$, along with a decomposition $\mat{M}=\mat{L}^T\mat{L}$, we have that $f$ is equivalent mod $\ideal$ to the SOS polynomial $\sum_{i=1}^{d}(\mat{L}\vec{z})_i^2$, which certifies its nonnegativity over $V_{\R}(\ideal)$. Altogether, we observe that implementing the SOS approach to certify polynomial nonnegativity over $V_{\R}(\ideal)$ amounts to solving an SDP feasibility problem. Moreover, one can obtain SDP-computable lower bounds for a polynomial $p$ over $V_\R(\ideal)$ by maximizing over scalars $\gamma$ for which nonnegativity of $p-\gamma$ over $V_\R(\ideal)$ can be certified using the SOS approach~\cite{ParriloTesis}. As SDPs can be computationally challenging to solve for large-scale problem instances, we present a more scalable alternative in \cref{sec: Deriving Spectral Relaxations}.

As a final remark, we note that obtaining a basis for $\R[\vec{x}]/\ideal$ and enabling operations modulo $\ideal$ requires the computation of a special set of generating polynomials for $\ideal$ called a Gröbner basis~\cite{RISC3775,idealsVarAlg}. While computing such a basis is difficult in the worst case, for many ideals appearing in applications a suitable choice is readily available (see \cref{sec: Experiments}) or can be found via existing tools from computational algebra~\cite{Singular,M2}. 

\paragraph{Notation} \label{sec: Notation}
 Bold symbols denote objects with multiple components, such as vectors, matrices, or tuples of polynomials. For a tuple $\vec{z} = (z_1, \dots, z_d) \in \R[\vec{x}]^d$ and an ideal $\ideal \subset \R[\vec{x}]$, we write $\vec{z} + \ideal\coloneq (z_1 + \ideal, \dots, z_d + \ideal)\in (\R[\vec{x}]/\ideal)^d$ to denote the corresponding tuple of equivalence classes mod $\ideal$. For $\vec{\alpha} \in \mathbb{N}^n$, we denote by $\vec{x}\coloneq x_1^{\alpha_1}\cdots x_n^{\alpha_n}$ the monomial with exponent vector $\vec{\alpha}$, and by $|\vec{\alpha}| := \sum_{i=1}^n \alpha_i$ its total degree. The set \(\mathbb{S}^d\) denotes the space of \(d \times d\) real symmetric matrices; \(\mathbb{S}_+^d\) and \(\mathbb{S}_{++}^d\) denote the cones of positive semidefinite and positive definite matrices, respectively.  We use \([s]\) to denote the index set $\{1, \dots, s\}$. The notation $\norm{\cdot}$ refers to the Euclidean norm on $\R^n$, and $\norm{\cdot}_F$ denotes the Frobenius norm of a matrix.

\section{Deriving Spectral Relaxations} \label{sec: Deriving Spectral Relaxations}

The POP~\cref{POP_init} can be rewritten as:
\begin{equation}\label{POP}
    \nu ~ := ~ \min_{\vec{x} \in \R^n} ~ p(\vec{x}) ~ \mathrm{s.t.} ~ \vec{x} \in V_\R(\ideal),
\end{equation}
where $\ideal\coloneq\langle g_1, \dots, g_{\ell}\rangle$ denotes the ideal generated by the constraints of \cref{POP_init}.  As discussed in the introduction, our goal is to formulate a spectral relaxation \cref{OP: SR} for \cref{POP}, thereby producing a lower bound for $\nu$ of the form $\lambda_{\min}(\mat{A},\mat{B})$, where $\mat{A}$ and $\mat{B}$ are symmetric matrices with $\mat{B}\succ 0$.  
%As discussed in the introduction, we wish to obtain a lower bound $$\lambda_{\min}(\mat{A},\mat{B})\leq\nu$, where $\mat{A}$ and $\mat{B}$ are conformal symmetric matrices and $\mat{B}\succ 0$, attained as the solution of a spectral relaxation \cref{OP: SR} for \cref{POP}. 

In \cref{sec: Identifying SRs}, we establish sufficient conditions on the polynomial $p$ and the ideal $\ideal$ that facilitate the formulation of spectral relaxations for \cref{POP}, and in \cref{sec: Initial SR}, we present a systematic approach for deriving one. \cref{sec: Hierarchy} then describes how to extend a given spectral relaxation for \cref{POP} into a hierarchy of such relaxations, resulting in a non-decreasing sequence of spectral lower bounds on $\nu$.
%In \cref{sec: Identifying SRs}, we identify sufficient conditions on the polynomial $p$ and the ideal $\ideal$ that facilitate the formulation of a spectral relaxation for \cref{POP}. In \cref{sec: Initial SR}, we present a systematic approach for deriving a spectral relaxation for \cref{POP}. Finally, \cref{sec: Hierarchy} outlines a method for deriving a hierarchy of spectral relaxations for \cref{POP} that yields a non-decreasing sequence of lower bounds on $\nu$.

\subsection{Conditions for the Existence of a  Spectral Relaxation} \label{sec: Identifying SRs}

%We establish conditions for the existence of spectral relaxations for the POP~\cref{POP} defined by the polynomial $p$ and the ideal $\ideal$, and we outline our general approach for identifying one. 

We begin by presenting a natural approach for identifying a spectral relaxation for the POP~\cref{POP} defined by the polynomial $p$ and the ideal $\ideal$.  Suppose we have a vector of polynomials $\vec{z}\in \R[\vec{x}]^d$ such that there exist Gram matrices $\mat{M}(1)\in \S^d_{++}$ and $\mat{M}(p)\in \S^d$ for the polynomials $1$ and $p$, respectively, with respect to the vector of quotient polynomials $\vec{z}+\ideal\in(\R[\vec{x}]/\ideal)^d$ (see \cref{sec: Background}).  Then for $\gamma\in\R$, we have the equivalence $p-\gamma\equiv_\ideal \vec{z}^T(\mat{M}(p)-\gamma\mat{M}(1))\vec{z}$, which yields the following spectral relaxation for \cref{POP}: \begin{equation}\lambda_{\min}(\mat{M}(p), \mat{M}(1))=\max_{\gamma\in\R} \;\; \gamma \; \; \mbox{s.t.}\;\; \mat{M}(p)-\gamma \mat{M}(1)\succeq 0.\label{OP: SR for POP}\end{equation}
Indeed, the condition $\mat{M}(p)-\gamma \mat{M}(1)\succeq0$ implies that $p-\gamma$ is equivalent mod $\ideal$ to a SOS polynomial, which certifies nonnegativity of $p-\gamma$ on $V_{\R}(\ideal)$ and establishes $\gamma$ as a lower bound on $p$ over the variety (see \cref{sec: Background}). The desired inequality, $\lambda_{\min}(\mat{M}(p), \mat{M}(1))\leq\nu$, follows. 

\begin{remark}\label{Rem: SOS relax} Building on the discussion in \cref{sec: Background}, we can associate to each spectral relaxation \cref{OP: SR for POP} a corresponding SOS-based SDP relaxation:
\begin{equation}\label{OP: SOS relax.}
    \max_{\gamma\in\R, \mat{Y}\in \S^d}\; \gamma \;\; \mbox{s.t.} \;\; p(\vec{x})-\gamma\equiv_\ideal \vec{z}^T\mat{Y} \vec{z}, \; \mat{Y}\succeq 0.
\end{equation}
Since any feasible $\gamma$ for \cref{OP: SR for POP} induces the feasible pair $(\gamma,\mat{M}(p)-\gamma\mat{M}(1))$ for \cref{OP: SOS relax.}, the SOS relaxation yields a tighter lower bound on $\nu$.  We reiterate, however, that the problem \cref{OP: SR for POP} can be solved using a generalized eigenvalue solver, whereas \cref{OP: SOS relax.} requires a more expensive SDP solver.
\end{remark}

The following example illustrates the approach for identifying spectral relaxations described above.

\begin{example} \label{ex: spectral bound over HC} \textit{Consider the ideal $\langle x_1^2-1, x_2^2-1,  x_3^2-1 \rangle$, whose variety $V_\R(\ideal)=\{\pm 1\}^3$ corresponds to the three-dimensional hypercube. The minimum value $\nu$ of the polynomial $p(x_1,x_2,x_3)=2x_1^2+x_1x_2-5x_2^2-2x_2x_3+3x_1-2x_3+12$ on $V_\R(\ideal)$ equals $1$.  To derive a spectral lower bound on $\nu$, consider the vector of polynomials $\vec{z}\coloneq(1,x_1,x_2,x_3)$ and note that $\vec{z}^T\vec{z}\equiv_\ideal 4$, so the normalized $4\times 4$ identity matrix $\frac{1}{4}\mat{I}$ is a positive definite Gram matrix for the polynomial $1$ with respect to $\vec{z}+\ideal$. Moreover, one can check that the matrix \begin{equation*}\mat{M}(p)\coloneq\frac{1}{4}{\begin{bmatrix}
9 & 6 & 0 & -4 \\
6 & 9 & 2 & 0 \\
0 & 2 & 9 & -4\\
-4 & 0 & -4 & 9\\
\end{bmatrix}}\end{equation*} is such that $p\equiv_\ideal\vec{z}^T\mat{M}(p)\vec{z}$. It follows that $\lambda_{\min}(\mat{M}(p),\frac{1}{4}\mat{I})\approx 0.74\leq\nu$. Note that $\mat{M}(p)$ is one of infinitely many Gram matrices for $p$ with respect to $\vec{z}+\ideal$.  We present methods for choosing $\mat{M}(p)$ in \cref{sec: Initial SR}.}
\end{example}

The above discussion raises the following fundamental question: given a polynomial $p$ and an ideal $\ideal$, when does there exist a vector $\vec{z}\in \R[\vec{x}]^d$ such that the polynomials $1$ and $p$ admit Gram matrices $\mat{M}(1)\in \S^d_{++}$ and $\mat{M}(p)\in \S^d$, respectively, with respect to $\vec{z}+\ideal$? To address this question, we present the following definition.
%The previous discussion raises the following fundamental question: given a polynomial $p$ and an ideal $\ideal$, when does there exist a vector $\vec{z}\in \R[\vec{x}]^d$ satisfying the conditions outlined above? To address this question, we introduce the following definition.
 
\begin{definition}\label{def: I-spherical polynomials} Let $\ideal\subset\R[\vec{x}]$ be an ideal. We say that a finite collection of polynomials $h_1, \dots, h_m \in \R[\vec{x}]$ is \emph{$\ideal$-spherical} if $\sum_{i=1}^m h_i^2\equiv_\ideal 1$. \end{definition}

A set of $\ideal$-spherical polynomials provides a non-trivial SOS representation for the polynomial $1$. This property is central to our development, as any such representation enables us to identify additional subsets of quotient polynomials with respect to which $1$ admits a positive definite Gram matrix.

Given a vector $\vec{h}=(h_1,\dots,h_m)\in\R[\vec{x}]^m$ of $\ideal$-spherical polynomials, observe that the $m$-dimensional identity is a Gram matrix for the polynomial $1$ with respect to $\vec{h}+\ideal$. Consequently, if there exists a Gram matrix $\mat{M}(p)$ for $p$ with respect to $\vec{h}+\ideal$, then the POP~\cref{POP} admits a spectral relaxation with associated bound $\lambda_{\min}(\mat{M}(p))\leq\nu$. This was the case for \cref{ex: spectral bound over HC}, with the choice of $\ideal$-spherical polynomials $\frac{1}{\sqrt{4}}\{1,x_1,x_2,x_3\}$. Nonetheless, for fixed  $p$ and $h_1,\dots,h_m$, the above provision does not typically hold. For instance, the same set of $\ideal$-spherical polynomials from \cref{ex: spectral bound over HC} does not admit a Gram matrix for the cubic polynomial $p(x_1,x_2,x_3)=x_1x_2x_3+3x_1x_2-4x_3$, as $p$ contains the term $x_1x_2x_3$ which cannot be expressed (modulo $\ideal$) as a linear combination of products of pairs from $\{1,x_1,x_2,x_3\}$. We turn next to a less restrictive condition.

For $k\in \N$, observe that $1\equiv_\ideal (\sum_{i=1}^m h_i^2)^k=(\vec{h}^{\otimes k})^T\vec{h}^{\otimes k}$, so the $m^k$-dimensional identity is a Gram matrix for $1$ with respect to the vector $\vec{h}^{\otimes k}+\ideal$.  To obtain a spectral relaxation for minimizing $p$ over $V_\R(\ideal)$, it thus suffices for there to exist some $k\in \N$ for which $p+\ideal$ lies in the subspace of $\R[\vec{x}]/\ideal$ spanned by $\vec{h}^{\otimes 2k}+\ideal$. Indeed, this would imply the existence of some Gram matrix $\mat{M}(p)\in\S^{m^k}$ for $p$ with respect to $\vec{h}^{\otimes k}+\ideal$, and $\lambda_{\min}(\mat{M}(p))$ would constitute a lower bound for $p$ over $V_\R(\ideal)$.  This discussion leads to our first result. We show that the approach described earlier in this subsection can be used to formulate spectral relaxations for minimizing $p$ over $V_\R(\ideal)$ whenever there exist $\ideal$-spherical polynomials $h_1, \dots, h_m\in \R[\vec{x}]$ such that $p+\ideal\in \text{Alg}_\ideal (h_1, \dots, h_m)$.

% Because $\ideal$-spherical polynomials are bounded by $1$ on $V_\R(\ideal)$, this condition provides an algebraic guarantee that $p$ is bounded over the variety. 

\begin{lemma} \label{lemma: sufficient condition}  Let $p\in\R[\vec{x}]$ and $\ideal\subset\R[\vec{x}]$ be an ideal. Suppose there exist $\ideal$-spherical polynomials $h_1,\dots, h_m\in \R[\vec{x}]$ such that $p+\ideal\in\text{Alg}_\ideal(h_1, \dots, h_m)$. Then, there exist a vector of polynomials $\vec{z}\in \R[\vec{x}]^d$ and matrices $\mat{M}(1), \mat{M}(p)\in\S^d$ satisfying $\mat{M}(1)\succ 0$, $1\equiv_\ideal \vec{z}^T\mat{M}(1)\vec{z}$, and $p \equiv_\ideal \vec{z}^T\mat{M}(p)\vec{z}$. \end{lemma}
\begin{proof}
    Let $h_1,\dots, h_m\in \R[\vec{x}]$ be $\ideal$-spherical polynomials satisfying $p+\ideal\in\mbox{Alg}_\ideal(h_1, \dots, h_m)$. Note that $\sum_{i=1}^m h_i^2-1\in \ideal$ if and only if $\sum_{i=1}^m (\frac{h_i}{\sqrt{2}})^2+(\frac{1}{\sqrt{2}})^2-1\in \ideal$, so we can assume without loss of generality that one of the $\ideal$-spherical polynomials $h_i$ is a positive constant. Under this premise, the linear subspaces of $\R[\vec{x}]/\ideal$ spanned by $\vec{h}^{\otimes k}+\ideal$, for $k\in \N$, are nested, and $p+\ideal$ is guaranteed to lie in the span of $\vec{h}^{\otimes 2k}+\ideal$ for some $k\in \N$. Statement 1 then follows from our earlier discussion. 
\end{proof}

We conclude by noting that the converse of \cref{lemma: sufficient condition} also holds. Indeed, given a vector $\vec{z}\in \R[\vec{x}]^d$ and Gram matrices $\mat{M}(1)\in\S^d_{++}$ and $\mat{M}(p)\in\S^d$ for the polynomials $1$ and $p$ with respect to $\vec{z}+\ideal$, the entries $\Tilde{z}_1, \dots, \Tilde{z}_d$ of the vector $\tilde{\vec{z}}\coloneq {\mat{M}(1)^{\frac{1}{2}}}\vec{z}$ form a collection of $\ideal$-spherical polynomials such that $p+\ideal\in \text{Alg}_\ideal(\Tilde{z}_1, \dots, \Tilde{z}_d)$. Consequently, the condition that $p$ is equivalent mod $\ideal$ to a polynomial function of some $\ideal$-spherical polynomials $h_1, \dots, h_m\in\R[\vec{x}]$ precisely characterizes the POPs for which we can derive spectral relaxations using the approach outlined in this subsection. In \cref{sec: SR amenable ideals}, we present conditions on an ideal $\ideal$ under which one can derive spectral relaxations for the minimization of any polynomial $p$ over $V_\R(\ideal)$. We also illustrate how these conditions are satisfied by several problem families arising in applications.

% While \cref{lemma: sufficient condition} provides conditions on the objective polynomial $p$ as well as the ideal $\ideal$ defined the constraints that yield a natural spectral relaxation, we present in \cref{sec: SR amenable ideals} conditions on $\ideal$ such that one can derive spectral relaxations for the minimization of any polynomial over $V_\R(\ideal)$.

\subsection{Constructing Effective Spectral Relaxations}\label{sec: Initial SR}

Consider a polynomial $p\in\R[\vec{x}]$ and an ideal $\ideal\subset\R[\vec{x}]$ such that $p+\ideal$ is in the subalgebra $\mbox{Alg}_{\ideal}(h_1, \dots, h_m)$ of $\R[\vec{x}]/\ideal$ for some $\ideal$-spherical polynomials $h_1, \dots, h_m$.  We build on the preceding subsection to obtain computationally effective spectral relaxations for the POP~\cref{POP} defined by $p$ and $\ideal$.  As the subspace $\mbox{Alg}_{\ideal}(h_1, \dots, h_m)$ of $\R[\vec{x}]/\ideal$ is typically high-dimensional (potentially even infinite-dimensional), a first step is to identify suitable low-dimensional subspaces of $\mbox{Alg}_{\ideal}(h_1, \dots, h_m)$ that enable the computation of tractable-sized Gram matrices for $p$ and $1$.  A natural grading of $\mbox{Alg}_{\ideal}(h_1, \dots, h_m)$ was already suggested in our discussion from \cref{sec: Identifying SRs}. Specifically, for $k\in \N$, we let \begin{equation*}\label{def: U subspace}
    U_{k, \ideal}(\vec{h})\coloneq\mbox{span}\left(\left\{h_{i_1}h_{i_2}\cdots h_{i_k}+\ideal: i_1, \dots, i_k\in [m]\right\}\right)
\end{equation*} be the subspace of $\R[\vec{x}]/\ideal$ spanned by the equivalence classes mod $\ideal$ of all products of exactly $k$ of the $\ideal$-spherical polynomials $h_1, \dots, h_m$, and denote by $d_k$ its dimension. Let $\kappa$ denote the smallest $k\in \N$ such that $p+\ideal$ lies in the subspace $U_{2k,\ideal}(\vec{h})$ (see the proof of \cref{lemma: sufficient condition}).  

For a polynomial $q\in\R[\vec{x}]$ such that $q+\ideal \in U_{2\kappa, \ideal}(\vec{h})$, consider the following Gram matrix for $q$ with respect to the vector of quotient polynomials $\vec{h}^{\otimes \kappa}+\ideal$:
\begin{equation} \label{eq: GM w.r.t. h kron kappa}
    \mat{Y}(q)\coloneq \mbox{argmin}_{\mat{Y}\in \S^{m^\kappa}} \; \norm{\mat{Y}}_F \; \mbox{s.t.}\; q\equiv_\ideal(\vec{h}^{\otimes \kappa})^T\mat{Y}\vec{h}^{\otimes \kappa}. 
\end{equation} 
The constraint above translates into finitely many linear equations on the entries of $\mat{Y}$, thus reducing \cref{eq: GM w.r.t. h kron kappa} to a standard linear least-squares problem. In light of \cref{sec: Identifying SRs}, the definition \cref{eq: GM w.r.t. h kron kappa} already yields a spectral relaxation for \cref{POP}. Indeed, the $m^\kappa$-dimensional identity is a positive definite Gram matrix for the polynomial $1$ with respect to $\vec{h}^{\otimes \kappa}+\ideal$, so the minimum eigenvalue of the matrix $\mat{Y}(p)$ defined by \cref{eq: GM w.r.t. h kron kappa} is a lower bound for $p$ over $V_\R(\ideal)$. However, the size of the vector $\vec{h}^{\otimes \kappa}+\ideal$ can be considerably larger than the dimension $d_\kappa$ of the subspace  $U_{\kappa, \ideal}(\vec{h})$ that it spans, so that the matrix $\mat{Y}(p)$ is unnecessarily large. On a related point, for $\kappa>1$ the redundancy in the entries of $\vec{h}^{\otimes \kappa}$ results in $\mat{Y}(p)$ having $0$ eigenvalues, which restricts the lower bound $\lambda_{\min}(\mat{Y}(p))$ to a non-positive value. Given these considerations, we propose an alternative method, which builds upon \cref{eq: GM w.r.t. h kron kappa} to produce a more effective, $d_\kappa$-dimensional spectral relaxation for the POP~\cref{POP}. \\

%In light of \cref{sec: Identifying SRs}, we are after a vector of polynomials $\vec{z}\in \R[\vec{x}]^d$ along with matrices $\mat{M}(1)$ and $\mat{M}(p)$ in $\S^d$ such that $\mat{M}(1)\succ0$, $1\equiv_\ideal\vec{z}^T\mat{M}(1)\vec{z}$ and $p\equiv_\ideal\vec{z}^T\mat{M}(p)\vec{z}$.

\noindent {\bf Method 1:} Fix a basis $\vec{z_\kappa}+\ideal$ for the subspace $U_{\kappa,\ideal}(\vec{h})$ of $\R[\vec{x}]/\ideal$, and consider the unique full column rank matrix $\mat{P}\in \R^{m^{\kappa}\times d_{\kappa}}$ such that $\mat{P}\vec{z_{\kappa}}\equiv_\ideal\vec{h}^{\otimes{\kappa}}$, whose $i$-th row contains the coordinates of the $i$-th entry of $\vec{h}^{\otimes{\kappa}}+\ideal$ in the basis $\vec{z_{\kappa}}+\ideal$. For $q\in\R[\vec{x}]$ such that $q+\ideal \in U_{2\kappa, \ideal}(\vec{h})$, we define: \begin{equation}\label{def: GM method 1}
    \mat{M_\kappa}(q)\coloneq \mat{P}^T\mat{Y}(q)\mat{P}.
\end{equation}
Combined with \cref{eq: GM w.r.t. h kron kappa}, $\mat{M_\kappa}(q)$ defines a Gram matrix for $q$ with respect to the basis $\vec{z_\kappa}+\ideal$. Therefore, if $\mat{M_\kappa}(1)\succ 0$, setting $\mat{M_\kappa}(1)$ and $\mat{M_\kappa}(p)$ as defined by \cref{def: GM method 1} produces a valid spectral relaxation for the POP~\cref{POP}, with $\lambda_{\min}(\mat{M_\kappa}(p),\mat{M_\kappa}(1))\leq \nu$. This constitutes our first proposed method for the construction of a spectral relaxation for~\cref{POP}. Note that the Gram matrix $\mat{M_\kappa}(q)$, defined by \cref{def: GM method 1} for $q+\ideal\in U_{2\kappa,\ideal}(\vec{h})$, is a linear function of $q+\ideal$. This property has several useful consequences, which we detail in the sequel (e.g., \cref{prop: Init translation invariance}). %As an example, we show later in \cref{prop: Init translation invariance} that it implies translation invariance of the bound $\lambda_{\min}(\mat{M_\kappa}(p),\mat{M_\kappa}(1))$ produced by Method 1, provided $\mat{M_\kappa}(1)\succ0$. Specifically, we have that $$\lambda_{\min}(\mat{M_{\kappa}}(p+\eta),\mat{M_{\kappa}}(1))=\lambda_{\min}(\mat{M_{\kappa}}(p),\mat{M_{\kappa}}(1))+\eta,\; \forall\; \eta\in \R.$$

\begin{example} \textit{We apply Method 1 to derive a spectral relaxation for minimizing $p(x_1,x_2)=x_1^3x_2-2x_1x_2^3+x_1x_2+x_2^4$ over the unit circle, which corresponds to $V_\R(\ideal)$ for $\ideal=\langle x_1^2+x_2^2-1\rangle$. Observe that the indeterminates $x_1$ and $x_2$ form a set of $\ideal$-spherical polynomials such that $p+\ideal\in U_{4,\ideal}(\vec{x})\leq\R[x_1,x_2]/\ideal$. We thus proceed to find the minimal Frobenius norm Gram matrices $\mat{Y}(1)$ and $\mat{Y}(p)$ for $1$ and $p$, respectively, with respect to $\vec{x}^{\otimes2}+\ideal$. This is achieved by computing the minimum-norm solution to the linear system that equates the coefficients of $(\vec{x}^{\otimes 2})^T \mat{Y} \vec{x}^{\otimes 2}+\ideal$ and $1+\ideal$ (respectively $p+\ideal$) with respect to the set of distinct monomials mod $\ideal$ that appear in the product $\vec{x}^{\otimes 4}$. We obtain: 
\begin{equation*}
   \mat{Y}(1)=\frac{1}{3}{\begin{bmatrix}
        3 & 0 & 0 &1 \\ 0 & 1 & 1 &0 \\ 0 & 1 & 1 & 0 \\ 1 & 0 & 0 & 3
    \end{bmatrix}}, \quad \quad  \mat{Y}(p)=\frac{1}{4}{\begin{bmatrix}
         0 & 2 & 2 &0 \\ 2 & 0 & 0 & -1 \\ 2 & 0 & 0 & -1 \\ 0 & -1 & -1 & 4
    \end{bmatrix}}.
\end{equation*}
Next, we compute the matrix $\mat{P}\in\R^{4\times 3}$ whose $i$-th row contains the coefficients of the $i$-th component of $\vec{x}^{\otimes 2}+\ideal$ with respect to the basis $(1, x_1x_2, x_2^2)+\ideal$ for $U_{2,\ideal}(\vec{x})$. Applying \cref{def: GM method 1}, we obtain the following Gram matrices for the polynomials $1$ and $p$ with respect to $(1, x_1x_2, x_2^2)+\ideal$: \begin{equation*}
    \mat{M_2}(1)=\frac{1}{3}{\begin{bmatrix} 3 & 0 & -2 \\ 0 & 4 & 0 \\ -2 & 0 & 4\end{bmatrix}}, \quad \quad \mat{M_2}(p)=\frac{1}{2}{\begin{bmatrix} 0 & 2 & 0 \\ 2 & 0 & -3 \\ 0 & -3 & 2
\end{bmatrix}}.
\end{equation*}
Note that $\mat{M_2}(1)\succ 0$, so Method 1 yields a valid spectral relaxation for the POP in question and establishes the lower bound $\lambda_{\min}(\mat{M_2}(p),\mat{M_2}(1))\approx -1.014$ on the minimum value of $p$ over the circle, which one can check is approximately $-0.532$.}
\end{example}
%{\color{red}Modify and elaborate on the next sentence as we discussed} This is achieved by fixing any basis for $U_{4,\ideal}(\vec{x})$ (e.g., $(1, x_1x_2, x_2^2, x_2^4, x_1x_2^3)+\ideal$) and computing the minimum-norm solution to the linear system that equates the coefficients of $(\vec{x}^{\otimes 2})^T \mat{Y} \vec{x}^{\otimes 2}+\ideal$ and $1+\ideal$ (respectively $p+\ideal$) with respect to the basis.

In the above example over the unit circle, the matrix $\mat{M_{\kappa}}(1)$ is positive definite. More generally, one can appeal to Proposition 2.2 in~\cite{lovitz2024hierarchyeigencomputationspolynomialoptimization}, along with some additional reasoning, to conclude that Method 1 always yields a positive definite Gram matrix for the polynomial $1$ when applied to the ideal $\ideal=\langle\sum_{i=1}^n x_i^2 -1 \rangle$ defining the unit sphere in $\R^n$ with $\ideal$-spherical polynomials $x_1, \dots, x_n$.  While this property holds for other ideals as well, there exist cases in which Method $1$ fails to return a positive definite Gram matrix for $1$, as illustrated in the following example.

\begin{example}\label{Ex: Method 1 fails} \textit{Consider the ideal $\ideal=\langle x_1^2+x_2^2-1,\;x_1x_2-\frac{1}{2},\; x_2^3+\frac{1}{2}x_1-x_2\rangle$ of $\R[x_1,x_2]$, for which $x_1$ and $x_2$ form a set of $\ideal$-spherical polynomials. Moreover, consider the subspace $U_{1,\ideal}(\vec{x})$ of $\R[\vec{x}]/\ideal$ with basis $\vec{x}+\ideal$. The Gram matrix $\mat{Y}(1)$ for $1$ with respect to $\vec{x}+\ideal$, as defined in \cref{eq: GM w.r.t. h kron kappa}, is given by:
\begin{equation*}
    \mat{Y}(1)=\frac{1}{3}{\begin{bmatrix}
        1 & 2 \\ 2 & 1
    \end{bmatrix}}.
\end{equation*}
The eigenvalues of the matrix $\mat{Y}(1)$ are $-\frac{1}{3}$ and $1$.  Method 1 yields $\mat{M_{1}}(1)=\mat{Y}(1)$ for the choice of basis $\vec{x}+\ideal$ for $U_{1,\ideal}(\vec{x})$, so we have an example for which the method fails to produce a valid spectral relaxation.}
\end{example}

As highlighted by this example, Method $1$ is not always suitable for obtaining a valid spectral relaxation. We now present a modification that addresses the underlying issue. \\

% Because the matrices $\mat{M_{\kappa}}(1)$ produced by Method 1 are not guaranteed to be positive definite, the method is not always suitable for deriving spectral relaxations. To address this issue, we propose the following variation.\\

\noindent {\bf Method 2:} Consider the basis $\vec{z_\kappa}+\ideal$ for $U_{\kappa,\ideal}(\vec{h})$, and recall that $\mat{P}$ denotes the unique matrix with full column rank such that $\mat{P}\vec{z_{\kappa}}\equiv_\ideal\vec{h}^{\otimes{\kappa}}$. For a polynomial $q\in \R[\vec{x}]$ such that $q+\ideal\in U_{2\kappa,\ideal}(\vec{h})$, we define:
\begin{equation}\label{def: GM method 2}
    \mat{{M}_\kappa}(q)\coloneq \mat{P}^T\mat{Y}(q-q_0)\mat{P}+q_0\mat{P}^T\mat{P}, 
\end{equation}
where $q_0$ denotes the constant term of the polynomial $q$, and $\mat{Y}(\cdot)$ is defined according to \cref{eq: GM w.r.t. h kron kappa}. One can check that $\mat{{M}_\kappa}(q)$ is again a Gram matrix for $q$ with respect to the basis $\vec{z_\kappa}+\ideal$. Further, we have that $\mat{{M}_\kappa}(1)=\mat{P}^T\mat{P}\succ0$. Consequently, the Gram matrices $\mat{{M}_\kappa}(p)$ and $\mat{{M}_\kappa}(1)$ defined by \cref{def: GM method 2} always lead to a valid spectral relaxation for the POP (\ref{POP}). For instance, applying Method 2 instead of Method 1 in \cref{Ex: Method 1 fails} (with all other choices unchanged) yields $\mat{M_1}(1)=\Id$, resolving the issues previously encountered. Finally, we note that by decoupling the computation of the Gram matrix for the constant term $q_0$ of $q$ from that of the remaining polynomial $q-q_0$, we preserve the linearity of $\mat{{M}_\kappa}(q)$ as a function of $q+\ideal$.   \\

We conclude this Section~by highlighting two features of the spectral bounds produced by the methods outlined above. We begin by showing that the bound $\lambda_{\min}(\mat{M_{\kappa}}(p),\mat{M_{\kappa}}(1))$ is translation invariant as a consequence of linearity.  

\begin{proposition}\label{prop: Init translation invariance} Let $p\in \R[\vec{x}]$ and $\ideal\subset\R[\vec{x}]$ be an ideal such that $p+\ideal$ lies in the subspace $U_{\kappa, \ideal}(\vec{h})$ for some integer $\kappa\in\N$ and some tuple $\vec{h}=(h_1, \dots, h_m)\in \R[\vec{x}]^m$ of $\ideal$-spherical polynomials. Furthermore, let $\mat{M_{\kappa}}(1)\succ 0$ and $\mat{M_{\kappa}}(p)$ be Gram matrices for $1$ and $p$, respectively, generated using either Method 1 or Method 2 for this setup. Then, the value $\lambda_{\min}(\mat{M_{\kappa}}(p),\mat{M_{\kappa}}(1))$ is translation invariant: $$\lambda_{\min}(\mat{M_{\kappa}}(p+q),\mat{M_{\kappa}}(1))=\lambda_{\min}(\mat{M_{\kappa}}(p),\mat{M_{\kappa}}(1))+\eta,$$ for all $\eta\in \R$ and $q\in \R[\vec{x}]$ such that $q\equiv_\ideal \eta$.
\end{proposition}
\begin{proof}
%The result follows from linearity of $\mat{M_\kappa}(q)$ in the coefficient vector $\vec{c}(q)$ of $q+\ideal$ with respect to any fixed basis $\vec{z_{2\kappa}}+\ideal$ for $U_{2\kappa, \ideal}(\vec{h})$. Indeed, we have that %For $\eta\in \R$ and $q\in \R[\vec{x}]$ with $q\equiv_\ideal \eta$, we have that $\mat{M_{\kappa}}(p+q)=\mat{M_{\kappa}}(p+\eta)$ by construction. The result then follows from linearity of $\mat{M_\kappa}(\cdot)$. Indeed, 
The result follows from linearity of $\mat{M_\kappa}(\cdot)$. Indeed, we have that 
    \begin{equation*}
        \begin{aligned}
            \lambda_{\min}(\mat{M_{\kappa}}(p+q),\mat{M_{\kappa}}(1))&= \max_{\gamma} \; \gamma \; \;\mbox{s.t.}\; \;\mat{M_\kappa}(p+\eta)-\gamma \mat{M_\kappa}(1)\succeq 0\\
            &= \max_{\gamma'} \;\; \gamma'+\eta \;\; \mbox{s.t.}\; \mat{M_\kappa}(p)-\gamma'\mat{M_\kappa}(1)\succeq 0\\
            &=\lambda_{\min}(\mat{M_{\kappa}}(p),\mat{M_{\kappa}}(1))+\eta,
        \end{aligned}
    \end{equation*}
where the first equality follows from the fact that $\mat{M_{\kappa}}(p+q)=\mat{M_{\kappa}}(p+\eta)$ for any $\eta\in \R$ and $q\in \R[\vec{x}]$ with $q\equiv_\ideal \eta$, and the second is obtained from the change of variables $\gamma'=\gamma+\eta$.
\end{proof}

Next, we show that the bound $\lambda_{\min}(\mat{M_\kappa}(p),(\mat{M_\kappa}(1))$ is independent of the choice of basis $\vec{z_\kappa}+\ideal$ for $U_{\kappa, \ideal}(\vec{h})$.

\begin{proposition}\label{prop: indep of basis kappa} Under the assumptions of \cref{prop: Init translation invariance}, the Gram matrices $\mat{M_{\kappa}}(1)\succ 0$ and $\mat{M_{\kappa}}(p)$ produced using either Method 1 or Method 2, are such that the value $\lambda_{\min}(\mat{M_{\kappa}}(p),\mat{M_{\kappa}}(1))$ is independent from the choice of basis $\vec{z_{\kappa}}+\ideal$ for $U_{\kappa, \ideal}(\vec{h})$. 
\end{proposition} 
\begin{proof} Let $\mat{M_{\kappa}}(\cdot)$ and $\mat{\hat{M}_{\kappa}}(\cdot)$ denote Gram matrices obtained from Methods 1 or 2 for two choices of basis $\vec{z_{\kappa}}+\ideal$ and $\vec{\hat{z}_{\kappa}}+\ideal$ for $U_{\kappa, \ideal}(\vec{h})$, respectively. Given the matrices $\mat{P}$ and $\mat{\hat{P}}$ such that $\mat{P}\vec{z_{\kappa}}\equiv_\ideal\mat{\hat{P}}\vec{\hat{z}_{\kappa}}\equiv_\ideal \vec{h}^{\otimes \kappa}$, one can check that $\mat{\hat{P}}=\mat{P}\mat{G}$, where $\mat{G}\in \R^{d_\kappa\times d_\kappa}$ denotes the unique invertible matrix satisfying $\vec{z_{\kappa}}\equiv_\ideal\mat{G}\vec{\hat{z}_{\kappa}}$. It follows that $\mat{\hat{M}_{\kappa}}(\cdot)=\mat{G}^T\mat{M_{\kappa}}(\cdot)\mat{G}$, which implies $\lambda_{\min}(\mat{\hat{M}_{\kappa}}(p),\mat{\hat{M}_{\kappa}}(1))=\lambda_{\min}(\mat{M_{\kappa}}(p),\mat{M_{\kappa}}(1))$.
\end{proof}

\subsection{A Hierarchy of Spectral Relaxations} \label{sec: Hierarchy}

In \cref{sec: Initial SR}, we describe how to derive a spectral relaxation for the POP~\cref{POP} associated to the polynomial $p$ and the ideal $\ideal$, given access to $\ideal$-spherical polynomials $h_1, \dots, h_m$ such that $p+\ideal\in\mbox{Alg}_{\ideal}(h_1, \dots, h_m)$.  Here we present a method to tighten a given spectral relaxation for \cref{POP} to obtain a spectral lower bound on the optimal value $\nu$ that is at least as good as the one at hand. By applying this procedure iteratively, one obtains a hierarchy of spectral relaxations for \cref{POP} that yields a non-decreasing sequence of lower bounds on its optimal value $\nu$. 

Fix $k\in\N$ such that $p+\ideal\in U_{2k,\ideal}(\vec{h})$, and suppose that we have access to Gram matrices $\mat{M_k}(1)$ and $\mat{M_k}(p)$ for the polynomials $1$ and $p$, respectively, relative to a fixed basis $\vec{z_k}+\ideal$ for the subspace $U_{k,\ideal}(\vec{h})$ of $\R[\vec{x}]/\ideal$. Because $U_{k+1,\ideal}(\vec{h})=\mbox{span}(\vec{h}\otimes\vec{z_k}+\ideal)$, and given a basis $\vec{z_{k+1}}+\ideal$ for $U_{k+1,\ideal}(\vec{h})$, there exists a unique matrix $\mat{L_{k+1}}\in\R^{md_k\times d_{k+1}}$ with full column rank such that $\mat{L_{k+1}}\vec{z_{k+1}}\equiv_\ideal\vec{h}\otimes\vec{z_k}$. We then have that \begin{equation*}\label{eq: step GM}
    p\equiv_\ideal \left(\vec{h}^T\vec{h}\right)\left(\vec{z_k}^T\mat{M_k}(p)\vec{z_k}\right) \equiv_\ideal \vec{z_{k+1}}^T\mat{L_{k+1}}^T\left(\Id_m\otimes\mat{M_k}(p)\right)\mat{L_{k+1}}\vec{z_{k+1}},
\end{equation*} 
%where the first equivalence follows from the polynomials $h_1,\dots, h_m$ being $\ideal$-spherical and $\mat{M_k}(p)$ serving as a Gram matrix for $p$ with respect to $\vec{z_k}+\ideal$, 
where the first equivalence follows from the assumptions that $\mat{M_k}(p)$ is a Gram matrix for $p$ with respect to $\vec{z_k}+\ideal$ and that the polynomials $h_1,\dots, h_m$ are $\ideal$-spherical, and the second follows from the definition of $\mat{L_{k+1}}$ together with the properties of the Kronecker product. Hence, we have the following Gram matrix for the polynomial $p$ with respect to the basis $\vec{z_{k+1}}+\ideal$:  $$\mat{M_{k+1}}(p)\coloneqq\mat{L_{k+1}}^T\left(\Id_m\otimes\mat{M_k}(p)\right)\mat{L_{k+1}}.$$ Similarly, we have that $\mat{M_{k+1}}(1)\coloneqq \mat{L_{k+1}}^T\left(\Id_m\otimes\mat{M_k}(1)\right)\mat{L_{k+1}}$ is a Gram matrix for the polynomial $1$ with respect to $\vec{z_{k+1}}+\ideal$. 

Based on the reasoning above, we propose the following algorithm for the construction of a hierarchy of spectral relaxations for the POP~\cref{POP}. As in \cref{sec: Initial SR}, we let $\kappa$ denote the smallest $k\in\N$ such that $p+\ideal\in U_{2k,\ideal}(\vec{h})$.

%\begin{algorith}[H]
\begin{algorithm}[H]
\caption{Iterative Construction of Spectral Relaxation Hierarchy}\label{alg: Iterative Construction}
\begin{algorithmic}
\State Initialize $k=\kappa$ 
\begin{enumerate}
    \item Compute $\mat{M_{\kappa}}(1)\succ 0$ and $\mat{M_{\kappa}}(p)$ as described in \cref{sec: Initial SR}. 
    \item Fix a basis $\vec{z_{k+1}}+\ideal$ for $U_{k+1, \ideal}(\vec{h})$.
    \item Compute $\mat{L_{k+1}}$ s.t. $\mat{L_{k+1}}\vec{z_{k+1}}\equiv_\ideal \vec{h}\otimes\vec{z_k}$. 
    \item Update $\mat{M_{k+1}}(1)=\mat{L_{k+1}}^T\left(\Id_m\otimes \mat{M_{k}}(1)\right)\mat{L_{k+1}}$
    \item Update $\mat{M_{k+1}}(p)=\mat{L_{k+1}}^T\left(\Id_m\otimes \mat{M_{k}}(p)\right)\mat{L_{k+1}}$
    \item Set $k\leftarrow k+1$ and return to step 2. 
\end{enumerate}
\end{algorithmic}
\label{iterative construction}
\end{algorithm}

\noindent We have the following result. 
%\begin{theorem} \label{theo: Main Theorem} Let $p\in\R[\vec{x}]$ and $\ideal\subset\R[\vec{x}]$ be an ideal such that $p$ is equivalent mod $\ideal$ to a polynomial function in some $\ideal$-spherical polynomials $h_1,\dots,h_m\in\R[\vec{x}]$. Without loss of generality, let $\kappa$ denote the first $k\in\N$ such that $p+\ideal\in U_{2k,\ideal}(\vec{h})$. Then, the matrices $\mat{M_k}(1)$ and $\mat{M_k}(p)$, for $k\geq \kappa$, generated by \cref{iterative construction} define a nested hierarchy of spectral relaxations for the POP (\ref{POP}), given by: $$\lambda_{\min}(\mat{M_k}(p),\mat{M_k}(1))=\max_{\gamma}  \gamma \; \mbox{s.t.} \; \mat{M_k}(p)-\gamma \mat{M_k}(1)\succeq 0.$$ In particular, we have that $\{\lambda_{\min}(\mat{M_k}(p),\mat{M_k}(1))\}_{k\geq\kappa}$ constitutes a non-decreasing sequence of spectral lower bounds for the minimum value of $p$ over $V_{\R}(\ideal)$. \end{theorem}

\begin{theorem} \label{theo: Main Theorem} 

Let $p\in\R[\vec{x}]$ and $\ideal\subset\R[\vec{x}]$ be an ideal, and let $\kappa\in\N$ denote the smallest integer $k\in\N$ for which $p+\ideal\in U_{2k,\ideal}(\vec{h})$. Then the matrices $\mat{M_k}(1)$ and $\mat{M_k}(p)$, for $k\geq \kappa$, generated by \cref{iterative construction} define the following hierarchy of spectral relaxations for the POP (\ref{POP}): $$\lambda_{\min}(\mat{M_k}(p),\mat{M_k}(1))=\max_{\gamma}  \gamma \; \mbox{s.t.} \; \mat{M_k}(p)-\gamma \mat{M_k}(1)\succeq 0.$$ In particular, we have that $\{\lambda_{\min}(\mat{M_k}(p),\mat{M_k}(1))\}_{k\geq\kappa}$ constitutes a non-decreasing sequence of spectral lower bounds for the minimum value of $p$ over $V_{\R}(\ideal)$. \end{theorem}

\begin{proof}
By construction, the matrices $\mat{M_k}(1)$ and $\mat{M_k}(p)$ generated with \cref{alg: Iterative Construction} are Gram matrices for the polynomials $1$ and $p$, respectively, with respect to some basis for the subspace $U_{k,\ideal}(\vec{h})$. Moreover, since the matrices $\mat{L_{k+1}}$ have full column rank, the positive definiteness of $\mat{M_k}(1)$ is preserved at each iteration, ensuring that the resulting matrices define a sequence of valid spectral relaxations for~\cref{POP}. Finally, for $\gamma \in \R$, the matrix $$\mat{M_{k+1}}(p)-\gamma\mat{M_{k+1}}(1) = \mat{L_{k+1}}^T(\mat{I}_m\otimes(\mat{M_k}(p)-\gamma\mat{M_k}(1)))\mat{L_{k+1}}$$ is positive semidefinite whenever $\mat{M_k}(p)-\gamma\mat{M_k}(1)\succeq 0$, as the matrices $\mat{A}$ and $\Id\otimes \mat{A}$ have the same eigenvalues. Thus, $\lambda_{\min}(\mat{M_{k}}(p),\mat{M_{k}}(1))\leq \lambda_{\min}(\mat{M_{k+1}}(p),\mat{M_{k+1}}(1)).$
\end{proof}

%\cref{alg: Iterative Construction} provides a way of constructing our spectral relaxation hierarchy iteratively, which is advantageous if one wishes to solve several levels of the hierarchy until a certain termination criterion is met. Nonetheless, the same hierarchy can be obtained by means of a non-iterative construction. Our next proposition illustrates how the Gram matrices $\mat{M_k}(1)$ and $\mat{M_k}(p)$ for a fixed level $k$ can be computed directly from the initial Gram matrices $\mat{M_{\kappa}}(1)$ and $\mat{M_{\kappa}}(p)$, without appealing to any of the calculations completed in the intermediate iterations $\kappa+1,\dots, k-1$.

\cref{alg: Iterative Construction} provides an iterative procedure for deriving a spectral relaxation hierarchy for the POP~\cref{POP}. Notably, the same hierarchy can also be obtained via a non-iterative construction. Specifically, the matrices $\mat{M_k}(1)$ and $\mat{M_k}(p)$ admit the following expressions  in terms of the initial Gram matrices $\mat{M_\kappa}(1)$ and $\mat{M_\kappa}(p)$:
\begin{equation} \label{eq: noniterative construction}
    \begin{aligned}
        &\mat{M_{k}}(p)=\mat{T_{k}}^T\left(\Id_m^{\otimes k-\kappa}\otimes \mat{M_{\kappa}}(p)\right)\mat{T_{k}} \;\; \\
        & \mat{M_{k}}(1)=\mat{T_{k}}^T\left(\Id_m^{\otimes k-\kappa}\otimes \mat{M_{\kappa}}(1)\right)\mat{T_{k}}\quad\quad
    \end{aligned}.
\end{equation}
Here, $\mat{T_k}$ is the matrix satisfying the equivalence $\mat{T_{k}}\vec{z_{k}}\equiv_{\ideal} \vec{h}^{\otimes (k-\kappa)}\otimes \vec{z_{\kappa}}$, which  exists and is unique because $\vec{z_{k}}+\ideal$ is a basis for the span of $\vec{h}^{\otimes (k-\kappa)}\otimes \vec{z_{\kappa}}+\ideal$. The expressions in \cref{eq: noniterative construction} may be verified by recursively solving for $\mat{M_k}(1)$ and $\mat{M_k}(p)$ using \cref{iterative construction}. 

%One can check that $\mat{T_k}= \prod_{i=1}^{k-\kappa} \Id_m^{\otimes k-\kappa-i}\otimes\mat{L_{\kappa+i}}$, where for each $i=1, \dots, k-\kappa$, the matrix $\mat{L_{\kappa+i}}$ is uniquely defined by $\mat{L_{\kappa+i}}\vec{z_{\kappa+i}}\equiv_\ideal\vec{h}\otimes\vec{z_{\kappa+i-1}}$. Solving the recursion prescribed by \cref{iterative construction}, then yields the following formula for the matrices $\mat{M_k}(1)$ and $\mat{M_k}(p)$ in terms of $\mat{T_k}$ and the initial Gram matrices $\mat{M_\kappa}(1)$ and $\mat{M_\kappa}(p)$: 

The non-iterative construction \cref{eq: noniterative construction} highlights some notable properties of our framework. One such feature is the independence of the lower bounds produced by our hierarchy from the choice of basis $\vec{z_k}+\ideal$ for $U_{k,\ideal}(\vec{h})$, for all $k\geq\kappa$. In other words, we have that the quality of the bounds produced by \cref{alg: Iterative Construction} relies only on the choice of $\ideal$-spherical polynomials $h_1, \dots, h_m$. 

\begin{corollary} Under the assumptions of \cref{theo: Main Theorem}, the lower bounds $\{\lambda_{\min}(\mat{M_k}(p),\mat{M_k}(1))\}_{k\geq\kappa}$ produced by \cref{alg: Iterative Construction} are independent of the choice of basis $\vec{z_k}+\ideal$ for $U_{k,\ideal}(\vec{h})$ for all $k\geq\kappa$. 
\label{cor: indep of basis k>=kappa}
\end{corollary} 
\begin{proof} 
%This property follows from \cref{prop: noniterative construction} and an argument analogous to that used in the proof of \cref{prop: indep of basis kappa}.
%We begin by arguing independence of the value $\lambda_{\min}(\mat{M_k}(p),\mat{M_k}(1))$ from the choice of basis for the subspace $U_{k,\ideal}(\vec{h})$. This was shown for $k=\kappa$ in \cref{prop: indep of basis kappa} and extends to $k>\kappa$ via \cref{prop: noniterative construction}. Indeed, let $\mat{T_k}$ and $\mat{\hat{T}_k}$ be the matrices provided by \cref{prop: noniterative construction} for two choices of basis $\vec{z_k}+\ideal$ and $\vec{\hat{z}_k}+\ideal$ for the subspace $U_{k,\ideal}(\vec{h})$, respectively, and let $\mat{G}\in \R^{d_k\times d_k}$ be the unique invertible matrix such that $\vec{z_k}\equiv_\ideal\mat{G}\vec{\hat{z}_k}$. One can check that $\mat{\hat{T}_{k}}={\mat{T_{k}}}\mat{G}$, so \cref{prop: noniterative construction} implies that $\mat{\hat{M}_k}(\cdot)=\mat{G}^T\mat{M_k}(\cdot)\mat{G}$, where $\mat{M_k}(\cdot)$ and $\mat{\hat{M}_k}(\cdot)$ denote the Gram matrices obtained at the $k$-th level of our hierarchy given the choices of basis $\vec{z_k}+\ideal$ and $\vec{\hat{z}_k}+\ideal$, respectively. The identity $\lambda_{\min}(\mat{M_k}(p),\mat{M_k}(1))=\lambda_{\min}(\mat{\hat{M}_k}(p),\mat{\hat{M}_k}(1))$ follows.

Independence of the value $\lambda_{\min}(\mat{M_k}(p),\mat{M_k}(1))$ from the choice of basis for the subspace $U_{k,\ideal}(\vec{h})$ was shown for $k=\kappa$ in \cref{prop: indep of basis kappa} and extends to $k>\kappa$ via \cref{eq: noniterative construction} and an analogous argument. % Indeed, given two choices of basis $\vec{z_k}+\ideal$ and $\vec{\hat{z}_k}+\ideal$ for the subspace $U_{k,\ideal}(\vec{h})$, respectively, one can check that the corresponding matrices $\mat{T_k}$ and $\mat{\hat{T}_k}$ provided by \cref{prop: noniterative construction} are such that $\mat{\hat{T}_{k}}={\mat{T_{k}}}\mat{G}$, where $\mat{G}\in \R^{d_k\times d_k}$ denotes the unique invertible matrix such that $\vec{z_k}\equiv_\ideal\mat{G}\vec{\hat{z}_k}$. Therefore, $\mat{\hat{M}_k}(\cdot)=\mat{G}^T\mat{M_k}(\cdot)\mat{G}$, and the result follows. 
It remains to show that $\lambda_{\min}(\mat{M_k}(p),\mat{M_k}(1))$, for $k>\kappa$, is independent from the choice of basis for the initial subspace $U_{\kappa,\ideal}(\vec{h})$. Let $\mat{M_k}(\cdot)$ and $\mat{\overline{M}_k}(\cdot)$ denote the Gram matrices produced by \cref{alg: Iterative Construction} given a fixed choice of basis $\vec{z_k}+\ideal$ for $U_{k,\ideal}(\vec{h})$ and the choices of basis $\vec{z_{\kappa}}+\ideal$ and $\vec{\overline{z}_{\kappa}}+\ideal$ for $U_{\kappa,\ideal}(\vec{h})$, respectively. One can check that $\mat{\overline{M}_{\kappa}}(\cdot)=\mat{R}^T\mat{M_{\kappa}}(\cdot)\mat{R}$ and $\mat{\overline{T}_{k}}=(\Id\otimes \mat{R}^{-1})\mat{T_k}$, where $\mat{T_{k}}\vec{z_{k}}\equiv_{\ideal} \vec{h}^{\otimes (k-\kappa)}\otimes \vec{z_{\kappa}}$, $\mat{\overline{T}_{k}}\vec{z_{k}}\equiv_{\ideal} \vec{h}^{\otimes (k-\kappa)}\otimes \vec{\overline{z}_{\kappa}}$, and $\mat{R}\in\R^{d_{\kappa}\times d_{\kappa}}$ denotes the unique invertible matrix satisfying $\vec{z_{\kappa}}\equiv_\ideal\mat{R}~\vec{\overline{z}_{\kappa}}$. Substituting into the expression for $\mat{\overline{M}_{k}}(\cdot)$ given by \cref{eq: noniterative construction} yields $\mat{\overline{M}_{k}}(\cdot)=\mat{M_k}(\cdot)$. The result follows. 
%Letting $\mat{T_k}$ and $\mat{\overline{T}_k}$ be such that $\mat{T_{k}}\vec{z_{k}}\equiv_{\ideal} \vec{h}^{\otimes (k-\kappa)}\otimes \vec{z_{\kappa}}$ and $\mat{\overline{T}_{k}}\vec{z_{k}}\equiv_{\ideal} \vec{h}^{\otimes (k-\kappa)}\otimes \vec{\overline{z}_{\kappa}}$, we have that $\mat{\overline{M}_{\kappa}}(\cdot)=\mat{R}^T\mat{M_{\kappa}}(\cdot)\mat{R}$ and $\mat{\overline{T}_{k}}=(\Id\otimes \mat{R}^{-1})\mat{T_k}$, where $\mat{R}\in\R^{d_{\kappa}\times d_{\kappa}}$ denotes the unique invertible matrix satisfying $\vec{z_{\kappa}}\equiv_\ideal\mat{R}~\vec{\overline{z}_{\kappa}}$. Substituting into the expression for $\mat{\overline{M}_{k}}(\cdot)$ given by \cref{eq: noniterative construction} yields $\mat{\overline{M}_{k}}(\cdot)=\mat{M_k}(\cdot)$. The result follows. 
\end{proof}

In \cref{sec: Initial SR}, we showed that the initial lower bound $\lambda_{\min}(\mat{M_{\kappa}}(p),\mat{M_\kappa}(1))$ produced by \cref{alg: Iterative Construction} is translation invariant. This property is inherited by the sequence $\{\lambda_{\min}(\mat{M_k}(p),\mat{M_k}(1))\}_{k\geq \kappa}$ due to the derivation of the spectral relaxations in the hierarchy from the initial relaxation via linear maps. 

\begin{corollary} Under the assumptions of \cref{theo: Main Theorem}, the sequence $\{\lambda_{\min}(\mat{M_{k}}(p),\mat{M_k}(1))\}_{k\geq \kappa}$ produced by \cref{alg: Iterative Construction} is translation invariant: $$\lambda_{\min}(\mat{M_{k}}(p+q),\mat{M_k}(1))=\lambda_{\min}(\mat{M_{k}}(p),\mat{M_k}(1))+\eta,$$ for any $\eta\in \R$ and $q\in \R[\vec{x}]$ such that $q\equiv_\ideal \eta$, and for all $ k\geq \kappa$.  
\end{corollary} 
\begin{proof}
    The result follows from \cref{prop: Init translation invariance} and equation \cref{eq: noniterative construction}. 
\end{proof}

\paragraph{Specialization to the case of homogenous minimization over the sphere.} We conclude this subsection by observing that our hierarchy of spectral relaxations produced by \cref{alg: Iterative Construction} (initialized with Method 1 in \cref{sec: Initial SR}) specializes to that of~\cite{lovitz2024hierarchyeigencomputationspolynomialoptimization} for homogeneous optimization over the sphere.  For a homogeneous polynomial $p\in \R[\vec{x}]$ of degree $2D$ and for $k\geq D$, the spectral bound proposed by the authors of~\cite{lovitz2024hierarchyeigencomputationspolynomialoptimization} for the minimization of $p$ over the Euclidean sphere $\mathcal{S}^{n-1}$ is given by $\lambda_{\min}(\mat{\tilde{M}_k}(p),\mat{\tilde{M}_k}(s))$, where $s(\vec{x})\coloneqq\left(\sum_{i=1}^n x_i^2\right)^D$ and the matrix $\mat{\tilde{M}_k}(q)$ for a homogeneous polynomial $q$ of degree 2D is defined as
\begin{equation*}\label{eq: Lovitz et al hierarchy}
    \mat{\tilde{M}_k}(q)\coloneqq\Pi_{n,k}\left(\Id_n^{\otimes k-D}\otimes\mat{{Y}}(q)\right)\Pi_{n,k}^T.
\end{equation*}
Here $\mat{{Y}}(q)$ denotes the matrix $\mat{{Y}}$ of minimal Frobenius norm satisfying the identity $q=\vec{x}^{\otimes D}\mat{{Y}}\vec{x}^{\otimes D}$, and $\Pi_{n,k}:(\R^n)^{\otimes{k}}\to \mbox{Sym}^k(\R^n)$ is the orthogonal projection onto the subspace of symmetric $k$-tensors.  To make the connection with our work, consider the ideal $\ideal=\langle \sum_{i=1}^n x_i^2 -1 \rangle$ defining the unit sphere $\mathcal{S}^{n-1}$, along with the natural choice of $\ideal$-spherical polynomials $x_1, \dots, x_n$. For a homogeneous polynomial $p\in \R[\vec{x}]$ of degree $2D$, observe that $p+\ideal\in U_{2k, \ideal}(\vec{x})$ for all $k\geq D$, and \cref{alg: Iterative Construction} can be applied with $\kappa=D$ to produce a sequence of spectral lower bounds for $p$ over $\mathcal{S}^{n-1}$. The matrices $\mat{\Tilde{M}_D}(s)$ and $\mat{\Tilde{M}_D}(p)$ coincide with the Gram matrices $\mat{M_D}(1)$ and $\mat{M_D}(p)$ for $1$ and $p$, respectively, that are produced by Method 1 with respect to the basis $\Pi_{n,D}\vec{x}^{\otimes D}+\ideal$ for $U_{D,\ideal}(\vec{x})$.  For $k>D$, one can verify that $\Pi_{n,k}=\mat{T_k}^T(\Id_{n}^{\otimes k-D} \otimes \mat{P}^T)$, where $\mat{P}$ and $\mat{T_k}$ are uniquely defined by the relations $\mat{P}\vec{z_D}\equiv_\ideal\vec{x}^{\otimes D}$ and $\mat{T_k}(\Pi_{n,k}\vec{x}^{\otimes k})\equiv_\ideal\vec{x}^{\otimes k-D}\otimes \vec{z_D}$ for $\vec{z_D}\coloneq \left( x_1^{\alpha_1} x_2^{\alpha_2} \cdots x_n^{\alpha_n} \right)_{\substack{\alpha \in \mathbb{N}^n, |\alpha| = D}}$. By \cref{def: GM method 1} and \cref{eq: noniterative construction}, the matrices $\mat{M_k}(1)$ and $\mat{M_k}(p)$ produced by \cref{alg: Iterative Construction} initialized with Method 1 for the choice of bases $\vec{z_D}+\ideal$ and $\Pi_{n,k}\vec{x}^{\otimes k}+\ideal$ for $U_{D,\ideal}(\vec{x})$ and $U_{k,\ideal}(\vec{x})$, respectively, coincide with $\mat{\Tilde{M}_k}(s)$ and $\mat{\Tilde{M}_k}(p)$. Because the bounds produced by our hierarchy of spectral relaxations are independent from this choice (see \cref{cor: indep of basis k>=kappa}), we conclude that implementing \cref{alg: Iterative Construction} with $\ideal$-spherical polynomials $x_1, \dots, x_n$ and Method 1 for initialization recovers the same sequence of bounds given by the construction of~\cite{lovitz2024hierarchyeigencomputationspolynomialoptimization}.

% reproducing the same bounds for the first method of initialization proposed in \cref{sec: Initial SR}.  

\section{Ideals Amenable to Our Framework for Spectral Relaxations} \label{sec: SR amenable ideals}

%In \cref{sec: Deriving Spectral Relaxations}, we establish a systematic approach for constructing spectral relaxations for the problem of minimizing a polynomial $p\in\R[\vec{x}]$ over the variety $V_\R(\ideal)$ of an ideal $\ideal\subset\R[\vec{x}]$, provided there exist $\ideal$-spherical polynomials $h_1, \dots, h_m\in \R[\vec{x}]$ such that $p+\ideal$ is in the subalgebra of $\R[\vec{x}]/\ideal$ generated by $h_1+\ideal, \dots, h_m+\ideal$. In \cref{sec: Characterization}, we examine the ideals $\ideal$ for which the entire quotient ring $\R[\vec{x}]/\ideal$ is generated by the equivalence classes mod $\ideal$ of some $\ideal$-spherical polynomials $h_1, \dots, h_m$. Evidently, these ideals $\ideal$ are precisely those for which our approach can be applied to derive spectral relaxations for any POP over $V_\R(\ideal)$. \cref{sec: Examples} discusses problem families for which the application of our method is particularly straightforward. %In \cref{sec: Characterization}, we present a characterization of such ideals in terms of an Archimedean condition prevalent in the polynomial optimization literature. In \cref{sec: Examples}, we discuss several problem families for which this condition is evidently satisfied. Finally, \cref{sec: Hom. Opt. over the sphere} highlights a connection between POPs that are amenable to our framework of spectral relaxations and homogeneous optimization over the sphere. 

In \cref{sec: Deriving Spectral Relaxations}, we identified a joint condition on a polynomial $p\in\R[\vec{x}]$ and an ideal $\ideal\subset\R[\vec{x}]$ that yields spectral relaxations for minimizing $p$ over the variety $V_\R(\ideal)$ and presented a systematic approach for their construction. In this section, we investigate ideals $\ideal$ for which our approach can be applied to derive spectral relaxations for minimizing any polynomial $p \in \R[\vec{x}]$ over $V_\R(\ideal)$. Specifically, \cref{sec: Characterization} gives a characterization of such ideals in terms of an Archimedean condition prevalent in the polynomial optimization literature and draws a connection between our framework and homogeneous optimization over the Euclidean sphere. \cref{sec: Examples} then discusses problem families and applications that are well-suited to our approach.

\subsection{An Archimedean Characterization} \label{sec: Characterization}
%\subsection{Ideals Amenable to Our Framework of Spectral Relaxations} \label{sec: SR amenable ideals}

We characterize those ideals which are amenable to our framework for spectral relaxations in terms of an Archimedean condition. This requires a few additional concepts from algebraic geometry. 

%We begin this Section~by showing that the ideals $\ideal\subset\R[\vec{x}]$ that are amenable to our framework of spectral relaxations are exactly those for which the following \emph{Archimedean} condition holds: there exists $N\in\N$ such that $N-\sum_{i=1}^n x_i^2$ is SOS mod $\ideal$ (see Proposition \EM{REF}). 

\begin{definition} The \emph{quadratic module} of an ideal $\ideal\subset\R[\vec{x}]$ is defined as: $$\Sigma(\ideal)\coloneqq \left \{\sum_{i=1}^s q_i^2+ \ideal:\; q_1,\dots, q_s \in \R[\vec{x}], ~ s \in \N \right \}.$$
\end{definition}

In words, the quadratic module of an ideal is the collection of all (equivalence classes of) polynomials that can be certified as being nonnegative over its variety based on a SOS decomposition (see \cref{sec: Background}). The Archimedean condition of interest pertains to the quadratic module of an ideal. 

\begin{definition} The quadratic module $\Sigma(\ideal)$ of an ideal $\ideal\subset\R[\vec{x}]$ is \emph{Archimedean} in $\R[\vec{x}]/\ideal$ if there exists $N\in \N \mbox{ s.t. } (N-\sum_{i=1}^n x_i^2)+\ideal \in \Sigma(\ideal)$.
\end{definition} 

The quadratic module of an ideal $\ideal$ being Archimedean provides an algebraic certificate for the compactness of $V_\R(\ideal)$, as the condition implies that $V_\R(\ideal)$ is contained in a Euclidean ball of radius $\sqrt{N}$. This condition is commonly employed in the literature to derive convergent hierarchies of SOS relaxations for POPs~\cite{Lasserre2001, ParriloTesis, Parrilo2003}. The next result shows that Archimedeanity of $\Sigma(\ideal)$ is equivalent to the entire quotient ring $\R[\vec{x}] / \ideal$ being generated as an algebra by the equivalence classes mod $\ideal$ of some $\ideal$-spherical polynomials $h_1,\dots,h_m \in\R[\vec{x}]$.  The previous condition characterizes the ideals $\ideal$ for which we can derive spectral relaxations for minimizing any polynomial $p \in \R[\vec{x}]$ over $V_{\R}(\ideal)$, as our method from \cref{sec: Deriving Spectral Relaxations} applies to a POP with objective $p$ over $V_\R(\ideal)$ precisely when $p+\ideal\in\mbox{Alg}_\ideal(h_1,\dots,h_m)$ for some $\ideal$-spherical polynomials $h_1,\dots,h_m \in\R[\vec{x}]$.

%Recall that our approach from Section~3 is applicable to a POP with constraint set $V_{\R}(\ideal)$ when the objective polynomial $p$ is equivalent mod $\ideal$ to a polynomial function in some $\ideal$-spherical polynomials $h_1,\dots,h_m \in\R[\vec{x}]$. The next result shows that Archimedeanity of $\Sigma(\ideal)$ is equivalent to the entire quotient ring $\R[\vec{x}] / \ideal$ being generated as an algebra by the equivalence classes mod $\ideal$ of some $\ideal$-spherical polynomials $h_1,\dots,h_m \in\R[\vec{x}]$. Evidently, the latter condition precisely describes the ideals $\ideal$ for which our approach can be applied to derive spectral relaxations for any POP over $V_\R(\ideal)$.

\begin{proposition} Let $\ideal\subset \R[\vec{x}]$ be an ideal. The quadratic module $\Sigma(\ideal)$ of $\ideal$ is  Archimedean in $\R[\vec{x}]/\ideal$ if and only if there exist $\ideal$-spherical polynomials $h_1,\dots, h_m\in\R[\vec{x}]$ such that $h_1+\ideal, \dots, h_m+\ideal$ generate $\R[\vec{x}]/\ideal$ as an algebra. 
\label{prop:SR_amenable_ideals}
\end{proposition}

\begin{proof} Suppose $\Sigma(\ideal)$ is Archimedean in $\R[\vec{x}]/\ideal$, and let $q_1,\dots,q_s\in\R[\vec{x}]$ and $N\in\N$ be such that $N-\sum_{i=1}^nx_i^2-\sum_{j=1}^s q_j^2 \in \ideal$. Then $\frac{1}{\sqrt{N}}\{x_1, \dots, x_n, q_1,\dots, q_s\}$ is a set of $\ideal$-spherical polynomials whose equivalence classes mod $\ideal$ generate $\R[\vec{x}]/\ideal$ as an algebra. Next, suppose there exist $\ideal$-spherical polynomials $h_1,\dots,h_m\in \R[\vec{x}]$ such that $\mbox{Alg}_\ideal(h_1, \dots, h_m)=\R[\vec{x}]/\ideal$, and consider the surjective $\R$-algebra homomorphism ${\phi}:\R[y_1,\dots,y_m]\to\R[\vec{x}]/\ideal$ defined by ${\phi}(y_i)\coloneqq h_i+\ideal$ for $i=1, \dots, m$. By the first isomorphism theorem, $\mathcal{J}\coloneq \mbox{ker}(\phi)$ is an ideal of $\R[\vec{y}]$ and the map $\psi:\R[\vec{y}]/\mathcal{J}\to \R[\vec{x}]/\ideal$, given by $\psi(y_i+\mathcal{J})\coloneq h_i+\ideal$ for $i=1, \dots, m$, defines an isomorphism of $\R$-algebras. Because $1-\left(\sum_{i=1}^m y_i^2\right)\in\mathcal{J}$, the quadratic module $\Sigma(\mathcal{J})$ of $\mathcal{J}$ is Archimedean in $\R[\vec{y}]/\mathcal{J}$. By an equivalent characterization of Archimedeanity (see Definition 5.2.1 and Corollary 5.2.4 in~\cite{marshall2008positive}), this implies the existence of some $N\in \N$ such that $N-\psi^{-1}(\sum_{i=1}^m x_i^2+\ideal)\in\Sigma(\mathcal{J})$. Since $\psi$ preserves SOS and constants, the result follows.
\end{proof}

Thus, Archimedeanity of the quadratic module $\Sigma(\ideal)$ of an ideal $\ideal \subset \R[\vec{x}]$ provides the key condition under which our framework can be applied to derive spectral bounds on any polynomial $p \in \R[\vec{x}]$ over $V_\R(\ideal)$. As elaborated in \cref{sec: Examples}, any POP with a bounded constraint set can be reformulated as a POP over the variety of an ideal with Archimedean quadratic module, enabling the formulation of spectral relaxations for general POPs over bounded algebraic varieties.

%\paragraph{An equivalent formulation with constraints given by subsets of the sphere}
\paragraph{An equivalent formulation over a subvariety of the sphere.} We conclude this subsection by highlighting a connection between our framework of spectral relaxations and homogeneous optimization over subsets of the Euclidean sphere. Specifically, given an ideal $\ideal\subset\R[\vec{x}]$ and $\ideal$-spherical polynomials $h_1, \dots h_m\in\R[\vec{x}]$ such that $\text{Alg}_\ideal(h_1, \dots, h_m)=\R[\vec{x}]/\ideal$, we show that the polynomial map defined by $h_1, \dots, h_m$ yields an isomorphism between the variety $V_\R(\ideal)$ and a subvariety of the sphere in $\R^m$. This enables the minimization of any polynomial over $V_\R(\ideal)$ to be reformulated as a problem in which a homogeneous polynomial is minimized over a subset of the sphere, in a manner that is compatible with our framework: applying the construction from \cref{sec: Deriving Spectral Relaxations} to either formulation yields the same sequence of spectral bounds on the underlying optimal value.  

%The resulting reformulation as a POP over the sphere differs from others found in the literature \TODO{cite} in that it arises from the algebraic structure of the quotient ring $\R[\vec{x}]/\ideal$, rather than relying on the introduction of additional variables. As a result, 

\begin{proposition}\label{prop:isomorphism_SV_of_Sphere}
Let $\ideal\subset\R[\vec{x}]$ be an ideal such that $\R[\vec{x}]/\ideal=\mbox{Alg}_\ideal(h_1, \dots, h_m)$ for some $\ideal$-spherical polynomials $h_1, \dots h_m\in \R[\vec{x}]$. The polynomial map $\Phi:V_{\R}(\ideal)\to\R^m$, given by $\Phi(\vec{x})\coloneq(h_1(\vec{x}), \dots, h_m(\vec{x}))$, defines an isomorphism between the variety $V_{\R}(\ideal)$ and a subvariety of the $(m-1)$-dimensional Euclidean sphere $\mathcal{S}^{m-1}$. 
\end{proposition}
\begin{proof} 
%The map $\psi(y_i+\mathcal{J})\coloneq h_i+\ideal$ for $i=1, \dots, m$, defines an isomorphism of $\R$-algebras $\psi:\R[\vec{y}]/\mbox{ker}({\phi})\to \R[\vec{x}]/\ideal$ between the quotient ring $\R[\vec{x}]/\ideal$ and that of the ideal $\mathcal{J}\coloneq \mbox{ker}(\phi)$ of $\R[\vec{y}]$ given by the kernel of the $\R$-algebra homomorphism ${\phi}:\R[\vec{y}]\to\R[\vec{x}]/\ideal$ induced by the polynomial map $\Phi$ (see the proof of \cref{prop:SR_amenable_ideals}). 
Consider the ideal $\mathcal{J}\subset\R[\vec{y}]$ and the $\R$-algebra isomorphism  $\psi:\R[\vec{y}]/\mathcal{J}\to \R[\vec{x}]/\ideal$ from the proof of \cref{prop:SR_amenable_ideals}. Naturally, there exist $q_1, \dots, q_n\in \R[\vec{y}]$ such that $\psi^{-1}(x_j+\ideal)=q_j+\mathcal{J}, \mbox{ for } j=1, \dots, n$. One can check that the polynomial map $\Gamma:V(\mathcal{J})\to \R^n$ defined by $\Gamma(\vec{y})\coloneq(q_1(\vec{y}), \dots, q_n(\vec{y}))$ is such that $\Phi\circ\Gamma$ and $\Gamma\circ\Phi$ correspond to the identity functions over $V_\R(\mathcal{J})$ and $V_{\R}(\ideal)$, respectively. Therefore, $\Phi$ is an isomorphism of varieties, and we have that $\Phi(V_{\R}(\ideal))=V_\R(\mathcal{J})$. Since $\sum_{i=1}^m h_i^2-1$ lies in $\ideal$, then $\sum_{i=1}^m y_i^2-1\in\mathcal{J}$, and we conclude that $\Phi(V_{\R}(\ideal))\subset\mathcal{S}^{m-1}$. 
\end{proof}

\begin{corollary} \label{cor: reformulation over the sphere} Under the assumptions of \cref{prop:isomorphism_SV_of_Sphere}, and without loss of generality, any POP over $V_\R(\ideal)$ can be reformulated as the minimization of a homogeneous polynomial of even degree over a subvariety of $\mathcal{S}^{m-1}$. Moreover, there exists a one-to-one correspondence between the hierarchies of spectral relaxations obtained from applying \cref{alg: Iterative Construction} to the reformulated POP and those obtained from applying it to the original problem.  
\end{corollary}

\begin{proof} Let $p\in \R[\vec{x}]$. Since $p+\ideal\in \mbox{Alg}_\ideal(h_1, \dots, h_m)$, there exist without loss of generality $\kappa\in \N$ and $\mat{M}(p)\in \S^{m^\kappa}$ such that $p\equiv_\ideal (\vec{h}^{\otimes{\kappa}})^T\mat{M}(p)\vec{h}^{\otimes{\kappa}}$ (see \cref{lemma: sufficient condition}). Minimizing $p$ over $V_{\R}(\ideal)$ is then equivalent to minimizing the homogeneous polynomial $q\coloneq (\vec{y}^{\otimes{\kappa}})^T\mat{M}(p)\vec{y}^{\otimes{\kappa}}\in\R[\vec{y}]$ over $\Phi(V_{\R}(\ideal))$, which, by \cref{prop:isomorphism_SV_of_Sphere}, corresponds to the subvariety of $\mathcal{S}^{m-1}$ defined by the ideal $\mathcal{J}\subset\R[\vec{y}]$ comprising all polynomial relations among the $\ideal$-spherical polynomials $h_1, \dots, h_m$ mod $\ideal$. 

The bijection between the spectral hierarchies produced by \cref{alg: Iterative Construction} when applied to the reformulated and original POPs is induced by the inverse $\psi^{-1}$ of the $\R$-algebra isomorphism $\psi$ from the proofs of \cref{prop:SR_amenable_ideals,prop:isomorphism_SV_of_Sphere}. Specifically, $\psi^{-1}$ preserves both the vector space and the ring structures of $\R[\vec{x}]/\ideal$: it maps each subspace $U_{k,\ideal}(\vec{h})\leq\R[\vec{x}]/\ideal$ isomorphically to a subspace $U_{k,\mathcal{J}}(\vec{y})\leq\R[\vec{y}]/\mathcal{J}$, transforms bases $\vec{z_k}+\ideal$ for $U_{k,\ideal}(\vec{h})$ into bases $\vec{w_k}+\mathcal{J}$ for $U_{k,\mathcal{J}}(\vec{y})$, and respects all the operations involved in the construction of a spectral hierarchy via \cref{alg: Iterative Construction}.
\end{proof}

Although \cref{cor: reformulation over the sphere} may not be practical -- applying our construction to the spherical reformulation of a POP over $V_\R(\ideal)$ requires one to compute the ideal $\mathcal{J}$ determined by all polynomial relations among the $\ideal$-spherical polynomials $h_1, \dots, h_m$ mod $\ideal$ -- it carries some interesting conceptual implications.  Specifically, the analysis of various properties of our spectral hierarchies, such as convergence rates of the spectral lower bounds to the optimal value of a POP, can be carried out by focusing solely on POPs with a very specific structure: those in which an even-degree homogeneous polynomial is minimized over a subvariety of the Euclidean sphere.  We do not investigate these directions further in the present paper.

\subsection{Examples} \label{sec: Examples}
The approach presented in \cref{sec: Deriving Spectral Relaxations} for constructing spectral relaxations for the problem of minimizing a polynomial $p\in\R[\vec{x}]$ over the variety $V_\R(\ideal)$ requires access to $\ideal$-spherical polynomials $h_1, \dots, h_m\in \R[\vec{x}]$ such that $p+\ideal$ is in the subalgebra of $\R[\vec{x}]/\ideal$ generated by $h_1+\ideal, \dots, h_m+\ideal$. We describe in this Section~some common problem families for which the $\ideal$-spherical polynomials $h_1, \dots, h_m$ are readily available.\\

%\noindent {\bf Optimization over Subvarieties of the Sphere:} A manifest candidate for the application of the approach outlined in \cref{sec: Deriving Spectral Relaxations} is polynomial optimization over subvarieties of the Euclidean sphere $\mathcal{S}^{n-1}$: \begin{equation}\label{OP: Opt over SV of Sphere}\begin{aligned} \min_{\vec{x}\in \R^n} p(\vec{x}) \; \mbox{s.t.} \;\; & q_1(\vec{x})=\dots=q_s(\vec{x})=0\\ & \sum_{i=1}^n x_i^2-1=0. \end{aligned}    \end{equation}Already, polynomial optimization over the entire sphere has applications in several areas~\cite{lovitz2024hierarchyeigencomputationspolynomialoptimization}. In \cref{sec: SR amenable ideals}, we established that any POP over the real variety of an ideal that is amenable to our framework of spectral relaxations admits a reformulation of the form (\ref{OP: Opt over SV of Sphere}), with $p$ being homogeneous of even degree.  \\ 

\noindent{\bf Optimization over the Sphere:}  A manifest candidate for the application of our framework from \cref{sec: Deriving Spectral Relaxations} is polynomial optimization over the Euclidean sphere. POPs over the sphere have applications in several domains \cite{Fawzi} and have inspired the development of tailored hierarchies for their approximation~\cite{Cristancho, lovitz2024hierarchyeigencomputationspolynomialoptimization}. \\

\noindent {\bf Combinatorial Optimization:} Consider a combinatorial optimization problem:
\begin{equation*}\label{OP: Opt over Hypercube}
\begin{aligned}
    \min_{\vec{x}\in \R^n} ~~~ p(\vec{x}) ~~~ \mbox{s.t.} ~~~ & q_1(\vec{x})=\dots=q_s(\vec{x})=0, ~ \vec{x}\in \{-1,1\}^n. 
\end{aligned}    
\end{equation*} 
Apparent instances of this formulation include the maximum-cut, boolean satisfiability, knapsack, stable set, and set covering problems. The constraint set here is the variety $V_{\R}(\ideal)$ of the ideal $\ideal=\langle x_1^2-1,\dots, x_n^2-1, q_1, \dots, q_s \rangle$, which is clearly amenable to spectral relaxations, as the polynomials $h_i\coloneqq \frac{x_i}{\sqrt{n}}$, for $i=1, \dots, n$, are $\ideal$-spherical and their equivalence classes mod $\ideal$ generate $\R[\vec{x}]/\ideal$ as an algebra.\\

% Consider a POP of the form
% \begin{equation}\label{OP: POP over M}
%     \min_{\mat{X}\in\R^{n\times r}} \; p(\mat{X}) \;\; \text{s.t.} \; \mat{X}\in\mathcal{M},
% \end{equation}
% % 

\noindent {\bf Optimization under Orthogonality Constraints:} Consider a POP in which a polynomial function of the entries of a matrix variable $\mat{X} \in \R^{n\times r}$ is minimized subject to one of the following constraint sets $\mathcal{M}$ -- the orthogonal group $\mathrm{O}_n \coloneq \{ \mat{X} \in \R^{n \times n} : \mat{X}^T\!\mat{X} =\mat{X}\mat{X}^T= \mat{I} \}$, the special orthogonal group $\mathrm{SO}_n \coloneq \{ \mat{X} \in \mathrm{O}_n : \det(\mat{X}) = 1 \}$, the Stiefel manifold $\mathrm{St}_{n,r} \coloneq \{ \mat{X} \in \R^{n \times r} : \mat{X}^T \mat{X} = \mat{I}_r \}$, or the oblique manifold (i.e., the set of ${n\times r}$ real matrices with unit-norm columns). Examples of POPs defined over such sets can be found in a wide range of applications \cite{Burer2023, edelmanORTH, procrustesProbs,  Hartley2013, Wen}. Observe that for all the constraint sets $\mathcal{M}$ described here, the matrices $\mat{X}\in\mathcal{M}$ have unit-norm columns. Therefore, after appropriate normalization, the entries of the variable $\mat{X}$ form a set of $\ideal$-spherical polynomials that generate $\R[\mat{X}]/\ideal$ as an algebra, thereby facilitating the derivation of our spectral relaxations. Further generalizations in which a polynomial objective $p\in\R[\mat{X}_1, \dots, \mat{X}_s]$ is optimized over $\mat{X}_1, \dots, \mat{X}_s\in\mathcal{M}$, with $\mathcal{M}$ being any of the sets discussed above, are also well-suited to our framework. Two notable examples are the little Grothendiek problems over $\mathrm{O}_n$ and $\mathrm{St}_{n,r}$, for which SDP relaxations have been studied~\cite{Bandeira2013ApproximatingTL}, along with applications to the orthogonal Procrustes problem~\cite{procrustesProbs}, global registration, and the common lines problem, which arises in 3D molecular reconstruction from cryo-EM data~\cite{Singer2011}. \\  %Examples of POPs defined over such sets $\mathcal{M}$ include the orthogonal Procrustes problem and some its variants~\cite{procrustesProbs}, rotation averaging problems arising in computer vision, robotics, and structural chemistry~\cite{Hartley2013}, and joint diagonalization tasks with applications to independent component analysis~\cite{Sato02122017}. Additional examples can be found in~\cite{Burer2023, jiang, edelmanORTH, Wen}. 

% For many problem families for which the choice of $\ideal$-spherical polynomials is not apparent, they can still be derived with minimal effort through a slight reformulation. This is the case for any POP with bounded constraint set. \\

\noindent {\bf Optimization over Bounded Algebraic Varieties:} 
Our framework from \cref{sec: Deriving Spectral Relaxations} can be applied to derive bounds on any POP over a bounded algebraic variety $V_\R(\ideal)$, provided access to the radius $R$ of some Euclidean ball containing $V_\R(\ideal)$. Indeed, by adding an extra variable $y$ to the problem, along with the constraint $R^2-\sum_{i=1}^nx_i^2=y^2$, one obtains an equivalent POP over the variety of an ideal $\ideal'$ of $\R[\vec{x},y]$ for which $\frac{1}{R}\{x_1, \dots, x_n, y\}$ constitutes a set of $\ideal'$-spherical polynomials whose equivalence classes mod $\ideal'$ generate $\R[\vec{x},y]/\ideal'$ as an algebra. In \cref{sec: BDTV}, we implement this strategy to lower bound the distance from a point to a variety.

%Note that all the problems mentioned above can be cast as polynomial optimization problems (POPs) over a subvariety of the sphere. In the next section, we demonstrate that this is a characteristic shared by all POPs that are suitable for spectral relaxations.

%%%%%%%%%%%%%%%%%%%%%%%%%%%%%%%%%%%%%%%%%%%%%%%%%%%%%%%%%%%%%%%%%%%%%%%%%%%%%%%%%%%%%%%%%%%%%%%%%%%%%%%%%%%%%%%%%%%%%%%%%%%%%%%%%%%%%%%%%%%%%%%%%%%%%%%%%%%%%%%%%%%%%%%%%%%%%%%%%%%

\section{Dual Perspective}\label{sec: Dual Perspective}

%Consider a POP~\cref{POP} defined by a polynomial $p\in\R[\vec{x}]$ and an ideal $\ideal\subset\R[\vec{x}]$ such that $p$ is equivalent mod $\ideal$ to a polynomial function in some $\ideal$ spherical polynomials $h_1, \dots, h_m \in\R[\vec{x}]$, and let $\kappa$ denote the first $k\in \N$ such that $p+\ideal$ lies in the linear subspace $U_{2k,\ideal}(\vec{h})$ of $\R[\vec{x}]/\ideal$ spanned by the equivalence classes mod $\ideal$ of the products of exactly $k$ of the polynomials $h_1, \dots, h_m$. 
%In \cref{sec: Derive Moment Relaxation}, we will show that (\ref{POP}) can be reformulated as a linear optimization problem over a set expressed in terms of $V_{\R}(\ideal)$ and the basis elements of the subspace $U_{2\kappa,\ideal}(\vec{h})$. 

In \cref{sec: Deriving Spectral Relaxations}, we identify algebraic structure that renders the POP~\cref{POP} defined by a polynomial $p\in\R[\vec{x}]$ and an ideal $\ideal\subset\R[\vec{x}]$ amenable to spectral relaxations. Given such structure, we find a subspace $U$ of  $\R[\vec{x}]/\ideal$ containing the quotient polynomial $p+\ideal$ and derive a sequence of spectral relaxations for \cref{POP} of the form:
\begin{equation}\label{OP: SOS side relax}
    \lambda_{\min}(\mat{M_k}(p), \mat{M_k}(1)) ~ = ~ \max_\gamma ~~~ \gamma ~~~ \mbox{s.t.} ~~~ \mat{M_k}(p)-\gamma \mat{M_k}(1)\succeq 0,
\end{equation} where for each $k$, the matrices $\mat{M_k}(p)$ and $\mat{M_k}(1)\succ 0$ are linear functions of the coefficients of $1+\ideal$ and $p+\ideal$, respectively, relative to any fixed basis for $U$. Consider the dual of \cref{OP: SOS side relax}: 
\begin{equation}\label{OP: Moment side relax}
\begin{aligned}
    \lambda_{\min}(\mat{M_k}(p), \mat{M_k}(1)) ~ = ~ \min_{\mat{X}\in \S^{d_k}} ~~~ \langle \mat{X},\mat{M_k}(p) \rangle ~~ \mbox{s.t.} ~~\langle \mat{M_k}(1),\mat{X}\rangle=1, ~ \mat{X}\succeq 0.
\end{aligned}
\end{equation}
In \cref{sec: Derive Moment Relaxation}, we will show that \cref{POP} can be reformulated as the minimization of a linear functional over a set $\mathcal{C}$ expressed in terms of the variety $V_{\R}(\ideal)$ and fixed basis elements for the subspace $U$. From this perspective, we will show that the dual problem \cref{OP: Moment side relax} can be viewed as a structured SDP in which the same linear objective is minimized but over a convex outer approximation of $\mathcal{C}$.  These outer approximations are highly structured; they are linear images of a base of the cone of positive-semidefinite matrices. 

\begin{definition} A set $\mathcal{S}\subset \R^d$ is a \emph{spectratope} if it is of the form: $$\mathcal{S}\coloneq \{\mathcal{L}(\mat{X}): \langle \mat{B},\mat{X} \rangle=b, \mat{X}\succeq 0 \}, $$
where $\mathcal{L}: \S^{D}\to \R^d$ is a linear map, $\mat{B}\in \S_{++}^{D}$ is positive definite, and $b>0$. 
\end{definition}
Observe that for $\vec{c}\in\R^d$, the minimum of the functional $f(\vec{w})=\langle \vec{c}, \vec{w} \rangle$ over the spectratope $\mathcal{S}\subset\R^\ell$ corresponds to the minimum generalized eigenvalue $\lambda_{\min}(\mathcal{L}^\star(\vec{c}), \frac{1}{b}\mat{B})$. Therefore, spectratopes correspond to projected spectrahedra over which linear minimization can be reduced to performing an eigencomputation. In \cref{sec: Lambda Bodies}, we demonstrate how the ideas developed in \cref{sec: Derive Moment Relaxation} can be used to obtain spectratope outer approximations of $V_{\R}(\ideal)$, which we contrast with previous outer approximations obtained from SOS methods. 

\subsection{Dual of a Spectral Relaxation}\label{sec: Derive Moment Relaxation}

%Consider the POP~\cref{POP} defined by a polynomial  $p\in\R[\vec{x}]$ and the ideal $\ideal\subset\R[\vec{x}]$, and assume that \cref{POP} is amenable to the framework of spectral relaxations presented in \cref{sec: Deriving Spectral Relaxations}. 

% We begin by observing that \cref{POP} can be reformulated as a linear optimization problem over a subset $\mathcal{C}\subset\R^{d_{2\kappa}}$, defined by the variety $V_{\R}(\ideal)$ and any fixed choice of basis for the subspace $U_{2\kappa,\ideal}(\vec{h})$. 

%Through this lens, the dual spectral relaxations \cref{OP: Moment side relax} for \cref{POP} correspond to a sequence of structured SDPs where the same linear function is minimized over a nested sequence of outer approximations of the set $\mathcal{C}$. \\

Assume that the POP~\cref{POP} is amenable to our framework of spectral relaxations, and let $\vec{h}\in\R[\vec{x}]^m$ be a vector of $\ideal$-spherical polynomials such that $p+\ideal$ lies in the $d_{2k}$-dimensional vector subspace $U_{2k, \ideal}(\vec{h})$ of $\R[\vec{x}]/\ideal$ for all integers $k\geq \kappa$, for some $\kappa\in \N$ (see \cref{sec: Initial SR}).  Fix a choice of basis $\vec{z_{2\kappa}}+\ideal$ for the subspace $U_{2\kappa,\ideal}(\vec{h})$ and denote by $\vec{c}(p)_{2\kappa}$ the coefficient vector of $p+\ideal$ in the basis $\vec{z_{2\kappa}}+\ideal$. Since $p\equiv_\ideal \langle \vec{c}(p)_{2\kappa}, \vec{z_{2\kappa}}\rangle$, we have that \cref{POP} is equivalent to
\begin{equation} \label{OP: POP reformulation kappa}
    \min_{\vec{y}\in \R^{d_{2\kappa}}} ~~~ \langle \vec{c}(p)_{2\kappa}, \vec{y} \rangle ~~~ \mbox{s.t.} ~~~ \vec{y}\in\mathcal{C}\coloneq\{\vec{z_{2\kappa}}(\vec{x}) : \vec{x}\in V_{\R}(\ideal)\}.
\end{equation} 
In words, \cref{POP} can be reformulated as a linear minimization problem over the set $\mathcal{C}$ given by the image of $V_{\R}(\ideal)$ under the polynomial map $\vec{x}\mapsto \vec{z_{2\kappa}}(\vec{x})$. 

Next, fix $k\geq\kappa$, and let $\mat{M_k}(p)$ and $\mat{M_k}(1)$ denote Gram matrices for $p$ and $1$, respectively, with respect to a basis $\vec{z_k}+\ideal$ for the subspace $U_{k,\ideal}(\vec{h})$ (obtained via \cref{alg: Iterative Construction}).  We show that the dual problem \cref{OP: Moment side relax} defined by $\mat{M_k}(p)$ and $\mat{M_k}(1)$ corresponds to a structured SDP in which the linear functional $\vec{c}(p)_{2\kappa}$ from \cref{OP: POP reformulation kappa} is minimized over a spectratope containing the set $\mathcal{C}$.  Given the fixed basis $\vec{z_{2\kappa}}+\ideal$ for $U_{2\kappa, \ideal}(\vec{h})$, the map $\mat{M_k}(\cdot)$ can be viewed as a linear function from $\R^{d_{2\kappa}}$ to the space of symmetric matrices $\S^{d_k}$, i.e., for $q+\ideal\in U_{2\kappa, \ideal}(\vec{h})$ the coefficients $\vec{c}(q)_{2\kappa}$ of $q+\ideal$ (relative to the basis $\vec{z_{2\kappa}}+\ideal$) are mapped to a Gram matrix $\mat{M_k}(q)$ for $q$ with respect to the basis $\vec{z_k}+\ideal$ for $U_{k,\ideal}(\vec{h})$.  This perspective allows us to interpret the dual spectral relaxation \cref{OP: Moment side relax} as the following linear minimization problem over a spectratope in $\R^{d_{2\kappa}}$: 
\begin{equation}\label{OP: Moment relax 2}
    \min_{\vec{y}\in \R^{d_{2\kappa}}} ~~~ \langle \vec{c}(p)_{2\kappa}, \vec{y}\rangle  ~~~ \mbox{s.t.} ~~~ \vec{y}\in \mathcal{S}_k\coloneqq\{\mat{M_k}^\star(\mat{X}): \langle \mat{M_k}(1), \mat{X}\rangle =1, \mat{X}\succeq 0\}, 
\end{equation}
where $\mat{M_k}^\star$ is the adjoint of the linear map $\mat{M_k}:\R^{d_{2\kappa}}\to\S^{d_k}$.  Note that the objective of \cref{OP: Moment relax 2} coincides with that of \cref{OP: POP reformulation kappa}. Our final result of the section establishes that the feasible sets $\mathcal{S}_k$ for \cref{OP: Moment relax 2} form a nested sequence of spectratope outer approximations to the constraint set $\mathcal{C}$ of \cref{OP: POP reformulation kappa}. 

%In particular, we have that $\vec{z_{2\kappa}}(\vec{x})\in \mathcal{S}_k$ for all $\vec{x}\in V_{\R}(\ideal)$ and $k\geq \kappa$. 

\begin{lemma} Let $\mathcal{C}$ and $\mathcal{S}_k$, for $k\geq \kappa$, be defined as in \cref{OP: POP reformulation kappa} and \cref{OP: Moment relax 2}. We have that  $\mathcal{C}\subset\dots\subset \mathcal{S}_{k+1}\subset\mathcal{S}_{k}\subset\dots$. 
\end{lemma}

\begin{proof} Let $\vec{y}\in \mathcal{C}$ and $\vec{x}\in V_{\R}(\ideal)$ be such that $\vec{y}=\vec{z_{2\kappa}}(\vec{x})$. Consider the matrix $\mat{X}=\vec{z_k}(\vec{x})\vec{z_k}(\vec{x})^T$. We have that $\mat{X}\succeq 0$ and $\langle\mat{M_k}(1), \mat{X}\rangle=1$, as $\vec{x}\in V_{\R}(\ideal)$ and $\mat{M_k}(1)$ is a Gram matrix for $1$ with respect to the basis $\vec{z_k}+\ideal$. Moreover, computing the $i$-th entry of the image $\mat{M_k}^\star(\mat{X})$ yields $$\langle\mat{M_k}^\star(\mat{X}),\vec{e_i}\rangle=\langle\mat{X}, \mat{M_k}(\vec{e_i})\rangle=z_{2\kappa,i}(\vec{x})=y_i,$$ where the second equality follows from the fact that $\mat{M_k}(\vec{e_i})$ is a Gram matrix for the $i$-th entry $z_{2\kappa,i}$ of  $\vec{z_{2\kappa}}$ with respect to $\vec{z_k}+\ideal$. We conclude that $\vec{y}=\mat{M_k}^\star(\mat{X})\in \mathcal{S}_k$, which implies that $\mathcal{C} \subset \mathcal{S}_k$.  To conclude that $\mathcal{S}_{k+1} \subset \mathcal{S}_k$, we observe that the bounds $\lambda_{\min}(\mat{M_k}(p), \mat{M_k}(1))$ are non-decreasing in $k$ for all $p$ (\cref{theo: Main Theorem}).
\end{proof}

% suppose for a contradiction that there exists $\vec{\hat{y}}\in\mathcal{S}_{k+1}\setminus\mathcal{S}_{k}.$ Because $\mathcal{S}_{k}$ is closed, the separation theorem guarantees the existence of a vector $\vec{c}\in\R^{d_{2\kappa}}$ such that $\langle \vec{c},\vec{\hat{y}} \rangle<\langle \vec{c}, \vec{y}\rangle$ for all $\vec{y}\in\mathcal{S}_{k}$. Consider the polynomial $p_{\vec{c}}\coloneq \langle \vec{c}, \vec{z_{2\kappa}}\rangle\in\R[\vec{x}]$. It follows from the reformulation \cref{OP: Moment relax 2} for \cref{OP: Moment side relax} that $$\lambda_{\min}(\mat{M_k}(p_{\vec{c}}), \mat{M_k}(1))=\min_{\vec{y}\in\mathcal{S}_k}\; \langle \vec{c},\vec{y} \rangle>\min_{\vec{y}\in\mathcal{S}_{k+1}}\; \langle \vec{c},\vec{y} \rangle=\lambda_{\min}(\mat{M_{k+1}}(p_{\vec{c}}), \mat{M_{k+1}}(1)),$$ which contradicts the fact that the values $\lambda_{\min}(\mat{M_k}(p_{\vec{c}}), \mat{M_k}(1))$ are non-decreasing in $k$ ( \cref{theo: Main Theorem}). 

\subsection{Outer Approximations to \texorpdfstring{$\mbox{conv}(V_{\R}(\ideal))$}{conv(V	extsubscript{R}(I))}} \label{sec: Lambda Bodies}

When the objective polynomial of the POP (\ref{POP}) is linear, it is natural to study the convex hull of its feasible set $V_\R(\ideal)$, denoted $\mbox{conv}(V_\R(\ideal))$, as linear optimization over a set is equivalent to linear optimization over its convex hull. As we will explain next, the ideas from \cref{sec: Derive Moment Relaxation} give rise to a nested sequence of spectratope outer approximations of $\mbox{conv}(V_\R(\ideal))$ for any ideal $\ideal\subset\R[\vec{x}]$ that is amenable to spectral relaxations (see Section \ref{sec: Characterization}).  Throughout, we let $\ideal\subset\R[\vec{x}]$ denote an ideal such that $\R[\vec{x}]/\ideal=\mbox{Alg}_{\ideal}(h_1, \dots, h_m)$ for some $\ideal$-spherical polynomials $h_1, \dots, h_m\in \R[\vec{x}]$. For simplicity, we also assume that the elements $1+\ideal, x_1+\ideal, \dots, x_n+\ideal$ of $\R[\vec{x}]/\ideal$ are linearly independent.\footnote{If not, then at least one of the indeterminates $x_i$ is expressible mod $\ideal$ as an affine function of the others, and one could work within a smaller polynomial ring.} We let $\tau\in\N$ denote the smallest integer $k$ for which the subspace $U_{2k,\ideal}(\vec{h})$ contains the equivalence class $f+\ideal$ for every linear polynomial $f\in\R[\vec{x}]$.

Given a basis $\vec{z_{2\tau}}+\ideal$ for $U_{2\tau,\ideal}(\vec{h})$, our spectral relaxation hierarchy produces a spectratope outer approximation $$\{\vec{z_{2\tau}}(\vec{x}) : \vec{x}\in V_{\R}(\ideal)\}\subset\mathcal{S}_k,$$ for each $k\geq \tau$ (see \cref{sec: Derive Moment Relaxation}). 
Here $\mathcal{S}_k$ is defined as in \cref{OP: Moment relax 2} for the appropriate linear map $\mat{M_k}:\R^{d_{2\tau}}\to\S^{d_k}$.  Because $U_{2\tau,\ideal}(\vec{h})$ contains all linear functions mod $\ideal$, there exists a unique matrix $\mat{P}\in\R^{n\times d_{2\tau}}$ such that $\vec{x}\equiv_\ideal \mat{P}\vec{z_{2\tau}}$. One can check that the spectratope $\mat{P}(\mathcal{S}_k)$ contains $V_\R(\ideal)$, thus providing a convex outer approximation.  In particular, if we choose $\vec{z_{2\tau}}$ to contain the indeterminates $x_1, \dots, x_n$, the map $\mat{P}$ is a coordinate projection and we have that $V_\R(\ideal)$ is contained in the spectratope given by the projection of $\mathcal{S}_k$ onto the corresponding $n$ coordinates. %In this setup, the associated linear map $\mat{M_k}(\cdot)$ sends the coefficient vector of $q+\ideal\in U_{2\tau,\ideal}(\vec{h})$, relative to $\vec{z_{2\tau}}+\ideal$, to a Gram matrix $\mat{M_k}(q)$ for $q$ with respect to some basis $\vec{z_k}+\ideal$ of $U_{2k,\ideal}(\vec{h})$.

\begin{example}\label{ex: outer approx}
    Consider the principal ideal $\mathcal{I}=\langle x^2 + 2y^2 - 2x^3 y - 2x^2 y^2 + 2x^4 y^2 - 4\rangle$, which is amenable to spectral relaxations with the $\ideal$-spherical polynomials $h_1=\frac{1}{2}(y-x^2y), \; h_2=\frac{1}{2}(x-x^2y), \; \mbox{and } h_3=\frac{1}{2}y$ generating $\R[\vec{x}]/\ideal$ as an algebra. \cref{fig: spectratope_OA} (left) displays the spectratope outer approximations to the convex hull of the variety $V_\R(\ideal)$ obtained from the first four levels of our hierarchy. 
\end{example}

\setlength{\textfloatsep}{0pt}
\begin{figure}[htbp]
  \centering
  \begin{subfigure}[b]{0.45\textwidth}
    \centering \includegraphics[width=\linewidth]{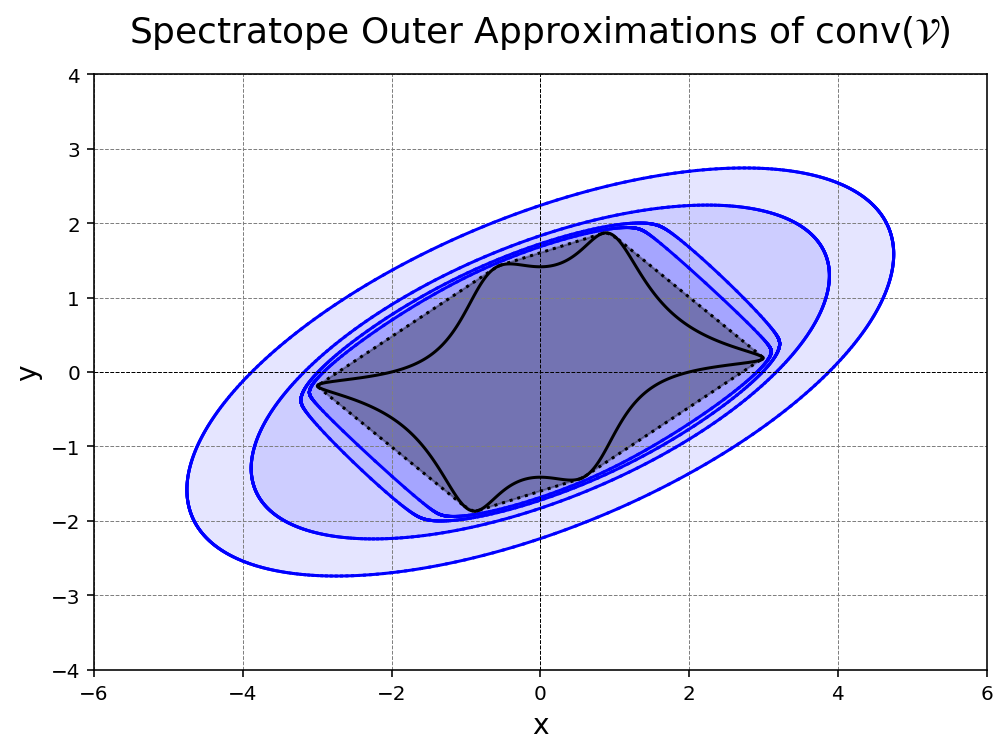} % square by making height=width
  \end{subfigure}\hfill
  \begin{subfigure}[b]{0.47\textwidth}
    \centering \includegraphics[width=\linewidth]{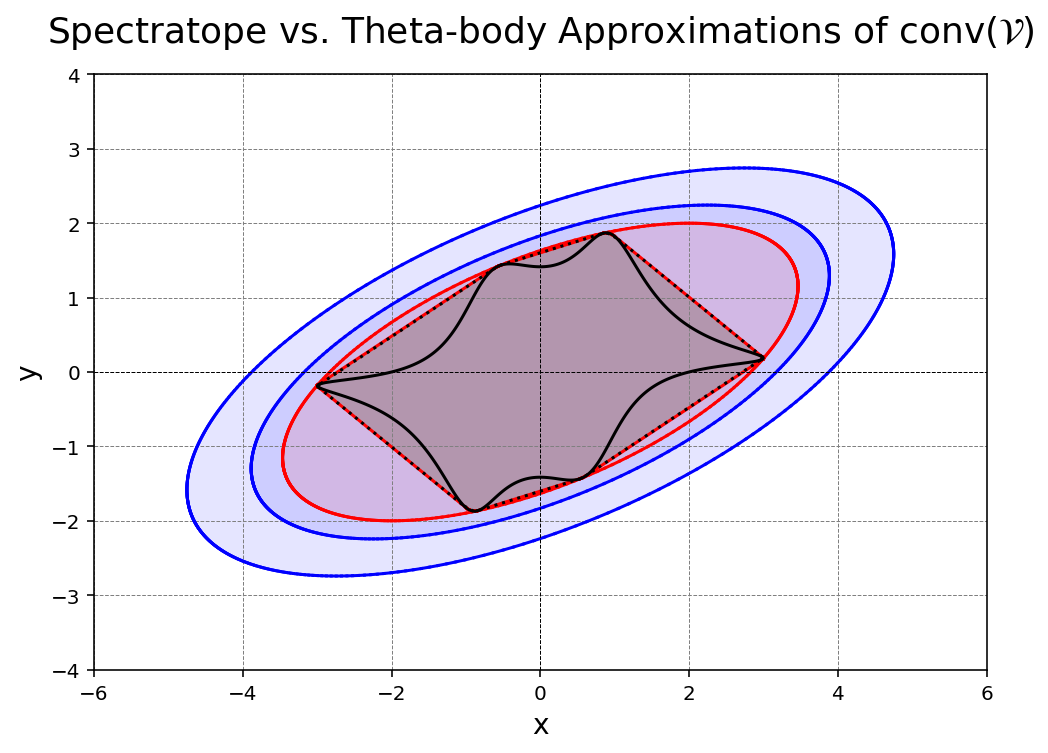}
  \end{subfigure}\hfill
  \captionsetup{font=footnotesize} 
  \captionsetup{skip=5pt}
  \caption{Outer approximations of the convex hull of $\mathcal{V}\coloneq V_\R( \langle x^2 + 2y^2 - 2x^3 y - 2x^2 y^2 + 2x^4 y^2 - 4 \rangle))$, shown as a grey region bounded by a black dotted line. \textit{Left:} Spectratopes corresponding to the first four levels of the spectral relaxation hierarchy. \textit{Right:} First- and second-level spectratope approximations~(blue), together with the theta bodies derived from the corresponding SOS relaxations~(red).}
  \label{fig: spectratope_OA}
\end{figure}

Recall that the closure of $\mbox{conv}(V_\R(\ideal))$ is characterized by the set of all closed halfspaces containing $V_\R(\ideal)$:
\begin{equation*}
\mbox{cl}(\mbox{conv}(V_\R(\ideal)))=\{\vec{x}\in \R^n: f(\vec{x}) \geq 0, ~\forall f \text{ affine and nonnegative on } V_\R(\ideal)\}.
\end{equation*}
Therefore, one can obtain tractable outer approximations to $\mbox{conv}(V_\R(\ideal))$ by restricting to subsets of affine functionals that can be efficiently certified as being nonnegative on $V_\R(\ideal)$. In~\cite{thetabodies}, the authors produce a sequence of semidefinite representable outer approximations to $\mbox{conv}(V_\R(\ideal))$, known as the \emph{theta bodies} of the ideal $\ideal$, by considering affine functionals that are SOS of a fixed degree modulo $\ideal$. Similarly, the spectratope outer approximations described in this work arise from considering affine functionals $f$ whose non-negativity over $V_\R(\ideal)$ can be certified by means of the inequality $\lambda_{\min}(\mat{M_k}(f),\mat{M_k}(1))\geq 0$. The bounds $\lambda_{\min}(\mat{M_k}(f),\mat{M_k}(1))$ are independent of the choice of basis with respect to which the matrices $\mat{M_k}(\cdot)$ are defined (\cref{cor: indep of basis k>=kappa}), but they depend on the choice of $\ideal$-spherical polynomials $\vec{h}$; the resulting spectratope outer approximations inherit these properties.  Building on \cref{Rem: SOS relax}, we note here that in general a spectratope approximation is weaker than the corresponding theta body approximation, as illustrated in Figure~\ref{fig: spectratope_OA}.

\FloatBarrier

\section{Numerical Experiments} \label{sec: Experiments}

We evaluate the performance of our spectral relaxation hierarchy on three different problems: the maximum-cut problem for a graph (\cref{sec: Max-Cut}), estimating the distance to an algebraic variety (\cref{sec: BDTV}), and computing the spectral norm of an order-$d$ tensor (\cref{sec:TSN}). Our primary comparison is with the corresponding SOS hierarchy, defined according to \cref{Rem: SOS relax}. While SOS relaxations are inherently tighter, our experiments show that spectral relaxations offer substantially better scalability, handling much larger problem instances.\footnote{Experiments were run on a 2.4 GHz quad‑core Intel Core i5 system with 8 GB of RAM. Optimization problems were solved using MOSEK \cite{mosek} via CVXPY \cite{cvxpy}. Eigenvalue problems were solved using SciPy’s \texttt{eigsh} routine, which wraps ARPACK’s implicitly‑restarted Lanczos method \cite{Lehoucq1998}.}

%producing initial bounds significantly faster 

\subsection{Maximum-Cut} \label{sec: Max-Cut} 

The maximum-cut problem on an undirected graph $\mathcal{G}$ with $n$ vertices and adjacency matrix $\mat{A} \in \mathbb{S}^n$ consists of finding a partition of its vertex set into two disjoint subsets so as to maximize the number of edges between them. It admits the following formulation as a POP over the hypercube $\{\pm 1\}^n$:

%\begin{equation}\label{OP: MAXCUT} \text{MAX-CUT}(\mathcal{G}) = -\min_{\vec{x}\in\R^n} \; \sum_{(i,j)\in \mathcal{E}}{\frac{x_ix_j-1}{2}} \quad \text{s.t.} \;\; x_i^2=1, \; i=1, \dots, n\end{equation}
\begin{equation}\label{OP: MAXCUT} \text{MAX-CUT}(\mathcal{G}) =\frac{1}{2} \left(\langle \mat{A}, \boldsymbol{1}\boldsymbol{1}^T\rangle -  \min_{\vec{x}\in\{\pm 1\}^n} \vec{x}^T\mat{A}\vec{x}\right) ,\end{equation} where $\boldsymbol{1}\in\R^n$ denotes the all ones vector. The ideal $\ideal=\langle x_1^2-1, \dots, x_n^2-1 \rangle$, with $V_\R(\ideal)=\{\pm 1\}^n$, is amenable to spectral relaxations via the $\ideal$-spherical polynomials $h_i=x_i/\sqrt{n}$, for $i=1,\dots,n$. Moreover, one can verify that the given generators for $\ideal$ form a Gröbner basis for the ideal, which yields the basis $\mathcal{B} = \{\vec{x}^{\vec{\alpha}} + \ideal : \vec{\alpha} \in \{0,1\}^n\}$ for the quotient ring $\R[\vec{x}]/\ideal$. This facilitates the practical construction of a spectral relaxation hierarchy for \cref{OP: MAXCUT} via \cref{alg: Iterative Construction}. 

Both initialization methods from \cref{sec: Initial SR} yield the Gram matrices $\mat{M_1}(1)=\frac{1}{n}\Id_n$ and $\mat{M_1}(p)=\mat{A}$ for the polynomials $1$ and $p=\vec{x}^T\mat{A}\vec{x}$ with respect to the basis $\vec{x}+\ideal$ for $U_{1,\ideal}(\vec{h})$, with $\mat{A}$ being the unique feasible choice for $\mat{M_1}(p)$. Given the natural choice of basis $(1,x_ix_j)_{1\leq i <j\leq n}+\ideal$ for $U_{2,\ideal}(\vec{h})$, the second level of the hierarchy also admits a closed-form expression: \begin{equation*}
\mat{M_2}(1)_{f_1,f_2} = 
{\begin{cases}
1/n \; , \;  & f_1=f_2=1 \\
2/n^2  \;, \;  & f_1=f_2\neq 1\\
0  \; , \;  &  \text{o.w.} \\
\end{cases}}, \;\; \mat{M_2}(p)_{f_1,f_2} = 
{\begin{cases}
2A_{ij}/n  \; , \;  &\{f_1,f_2\}=\{1,x_ix_j\} \\
A_{kl}/n  \;, \;  & f_1=x_ix_k\neq x_ix_l=f_2 \\
0  \; , \;  &  \text{o.w.} \\
\end{cases}}
\end{equation*} 
Observe that the matrices $\mat{M_2}(\cdot)$ are sparse even when $\mat{A}$ is not. Such sparsity is a common feature of our spectral relaxations and it can be exploited by standard solvers to accelerate computations.

\cref{tab:sr_vs_sos_time,tab:sr_vs_sos_bounds} compare the first and second levels of our spectral hierarchy with the corresponding SOS relaxations for the maximum-cut problem on Erdös-Rènyi random graphs of varying sizes. As anticipated, spectral relaxations accommodated larger problem instances than their SOS counterparts and generated bounds orders of magnitude faster. Despite this significant computational advantage, the difference in bound quality remained modest for this example. We compute relative differences between our spectral bounds and the first SOS bound, calculated as $(\mbox{SR}_{\mbox{bound}}-\mbox{SOS}_{\mbox{bound}})/\mbox{SOS}_{\mbox{bound}}$, over $100$ random instances of \cref{OP: MAXCUT} and for graph sizes $n$ for which a first SOS bound could be computed within reasonable time. \cref{tab:sr_rel_diff} reports the average relative difference for each problem size and spectral relaxation level. Notably, all values remain below 2.51\% and decrease as $n$ increases.

Alternative approaches to address the limitations of SOS methods — specifically the DSOS and SDSOS hierarchies \cite{Ahmadi2019} — yield the trivial bound $\frac{1}{2}\langle \mat{A}, \boldsymbol{1}\boldsymbol{1}^T\rangle$ at the first level of those hierarchies. (This issue persists even when these methods are adapted to the quotient ring perspective that we have adopted in this paper, i.e., when the positive-semidefinite constraint in \cref{Rem: SOS relax} is replaced by a diagonally dominant or scaled diagonally dominant constraint.)
% when these methods are naively adapted to our ideal-description-invariant framework by replacing the PSD constraint at each level of the corresponding SOS hierarchy with a diagonally dominant or scaled diagonally dominant constraint. % Such adaptations do not permit a fair comparison with our approach.

\setlength{\textfloatsep}{0pt}
\begin{table}[H]
\centering
\small
\setlength{\tabcolsep}{3pt}
\begin{tabular}{|l|r|r|r|r|r|r|r|r|}
\hline
Method & $n=25$ & $n=50$ & $n=100$ & $n=250$ & $n=500$ & $n=1250$ & $n=1500$ & $n=40000$ \\
\hline
SR, $k=1$   & 0.002 & 0.003 & 0.002 & 0.009 & 0.039   & 0.430    & 0.624   & 9892 \\
SR, $k=2$   & 0.052 & 0.235 & 1.583 & 25.38  & 320.4  & 79278   & X      & X \\
SOS, $k=1$  & 0.362 & 1.494 & 11.69  & 7127  & X      & X       & X      & X \\
SOS, $k=2$  & 23332 & X     & X     & X     & X      & X       & X      & X \\
\hline
\end{tabular}
\captionsetup{font=footnotesize} 
\captionsetup{skip=5pt}
\caption{Time (seconds) to compute the 1st and 2nd spectral (SR) and SOS bounds on the maximum-cut value of an Erdös-Rènyi random graph of size $n$ with edge probability $\rho=0.7$. An X denotes exceeded memory limits.}
\label{tab:sr_vs_sos_time}
\end{table}
%\vspace{-7mm}

\begin{table}[H]
\centering
\small
\setlength{\tabcolsep}{3pt}
\begin{tabular}{|l|r|r|r|r|r|r|r|r|}
\hline
Method & $n=25$ & $n=50$ & $n=100$ & $n=250$ & $n=500$ & $n=1250$ & $n=1500$ & $n=40000$ \\
\hline
SR, $k=1$   & 257.7  & 1022 & 3930   & 23600  & 92412   & 567449  & 814364  & 5.64e+08 \\
SR, $k=2$   & 255.9  & 1020 & 3924   & 23589  & 92398   & 567423  & NA      & NA      \\
SOS, $k=1$  & 251.4  & 1008 & 3889   & 23484  & NA      & NA      & NA      & NA      \\
SOS, $k=2$  & 244.0  & NA     & NA     & NA     & NA      & NA      & NA      & NA      \\
\hline
\end{tabular}
\captionsetup{font=footnotesize} 
\captionsetup{skip=5pt}
\caption{Upper bounds on the maximum-cut value of an Erdős–Rényi random graph of size $n$ with edge probability $\rho = 0.7$, computed using the first two levels of the spectral relaxation (SR) hierarchy and the corresponding SOS hierarchy.} 
\label{tab:sr_vs_sos_bounds}
\end{table}
%\vspace{-7mm}

\begin{table}[H]
\centering
\small
\setlength{\tabcolsep}{3pt}
\begin{tabular}{|l|r|r|r|}
\hline
Method & $n=25$ & $n=50$ & $n=100$ \\
\hline
SR, $k=1$   & 2.51\% & 1.74\% & 1.16\% \\
SR, $k=2$   & 1.86\% & 1.42\% & 1.01\% \\
\hline
\end{tabular}
\captionsetup{font=footnotesize} 
\captionsetup{skip=5pt}
\caption{Relative difference (\%) between the 1st and 2nd spectral bounds and the 1st SOS bound on the maximum-cut value of Erdős–Rényi random graphs (with edge probability $\rho = 0.7$) of size $n$, averaged over 100 instances.}
\label{tab:sr_rel_diff}
\end{table}
%\vspace{-7mm}

\FloatBarrier
\subsection{Bounding the distance to a variety} \label{sec: BDTV}
For a bounded variety $V_\R(\ideal)\subset\R^n$ and a fixed point $\hat{\vec{x}}\in \R^n\setminus V_\R(\ideal)$, consider the problem of lower bounding the distance $d(\hat{\vec{x}},V_\R(\ideal))$ between $\hat{\vec{x}}$ and $V_\R(\ideal)$. Problems of this form arise in engineering applications such as robust bifurcation analysis and robust control, where guarantees of a minimal distance between the nominal parameters of a given system and a problematic set of values are of interest (see~\cite[p.~93]{ParriloTesis},~\cite[p.~81]{blekherman2013semidefinite} and references therein). Given access to the radius $R$ of a Euclidean ball containing $V_\R(\ideal)$, one can derive lower bounds on $d(\hat{\vec{x}},V_\R(\ideal))$ by applying \cref{alg: Iterative Construction} to the POP:
\begin{equation*}
\begin{aligned}
    d(\hat{\vec{x}},V_\R(\ideal))=\min_{(\vec{x},y)\in \R^{n+1}} \; \norm{\vec{x}-\hat{\vec{x}}}^2 \; \mbox{s.t.} \; \; & q_1(\vec{x})= \dots =q_s(\vec{x})=0\\
    & R^2-\norm{\vec{x}}^2-y^2=0,
\end{aligned}
\end{equation*}
with the choice of $\ideal$-spherical polynomials $\frac{1}{\sqrt{R^2+1}}\{1, x_1,\dots,x_n,y\}$.

For illustration, we compute lower bounds on the distance from six points in $\R^2$ to the variety $\mathcal{V}\coloneq V_\R(\ideal)$ of the principal ideal $\ideal=\langle 8(x^{4} + y^{4}) - 10(x^{2} + y^{2}) + 6 x^{2} y^{2} + 3 \rangle$, which is contained in the ball of radius $R=\sqrt{1.5}$. \cref{fig:BDTV} shows the best bounds obtained from \cref{alg: Iterative Construction} (initialized using Methods 1 and 2) and from the corresponding SOS hierarchy after $0.01, \;0.02,\text{ and } 0.04$ seconds of runtime. When available, SOS bounds are typically tighter than spectral bounds for the same runtime budget, as achieving comparable quality with spectral relaxations requires solving eigenvalue problems of much larger dimension. Nevertheless, within just $0.01$ seconds, both spectral hierarchies produced meaningful bounds for five out of the six points, whereas the SOS hierarchy failed to produce any. % (e.g., for $\hat{\vec{x}}=(0.3,0.8)$ matching the first meaningful SOS bound on $d(\hat{\vec{x}},\mathcal{V})$ required advancing to levels $35$ and $38$ of our first and second spectral hierarchies, respectively)

\setlength{\textfloatsep}{4pt}
\begin{figure}[H]
  \centering
  \begin{subfigure}[b]{0.33\textwidth}
    \centering \includegraphics[width=\linewidth,height=\linewidth]{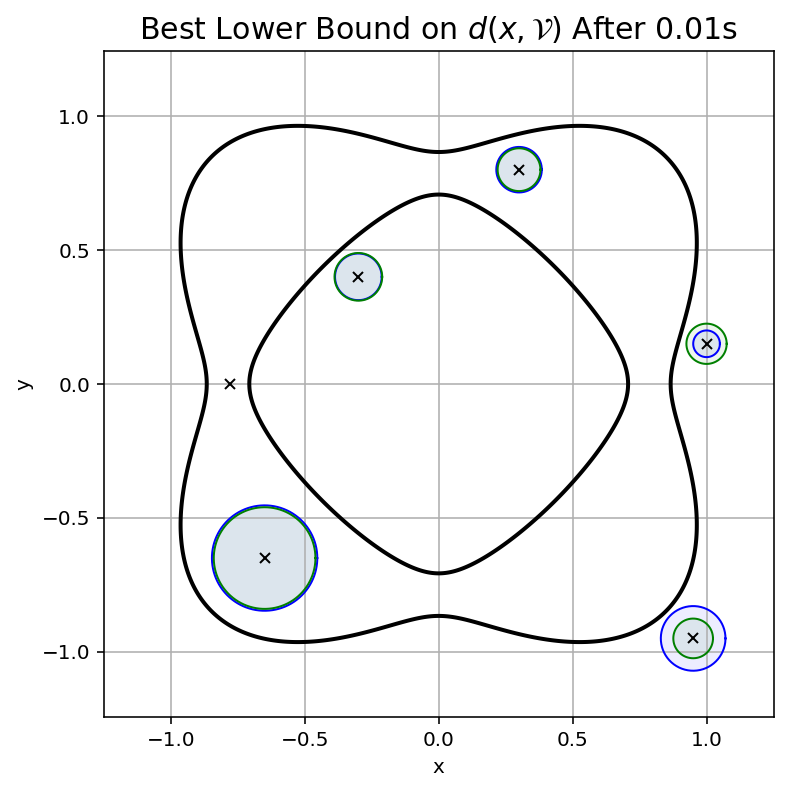} % square by making height=width
  \end{subfigure}\hfill
  \begin{subfigure}[b]{0.33\textwidth}
    \centering \includegraphics[width=\linewidth,height=\linewidth]{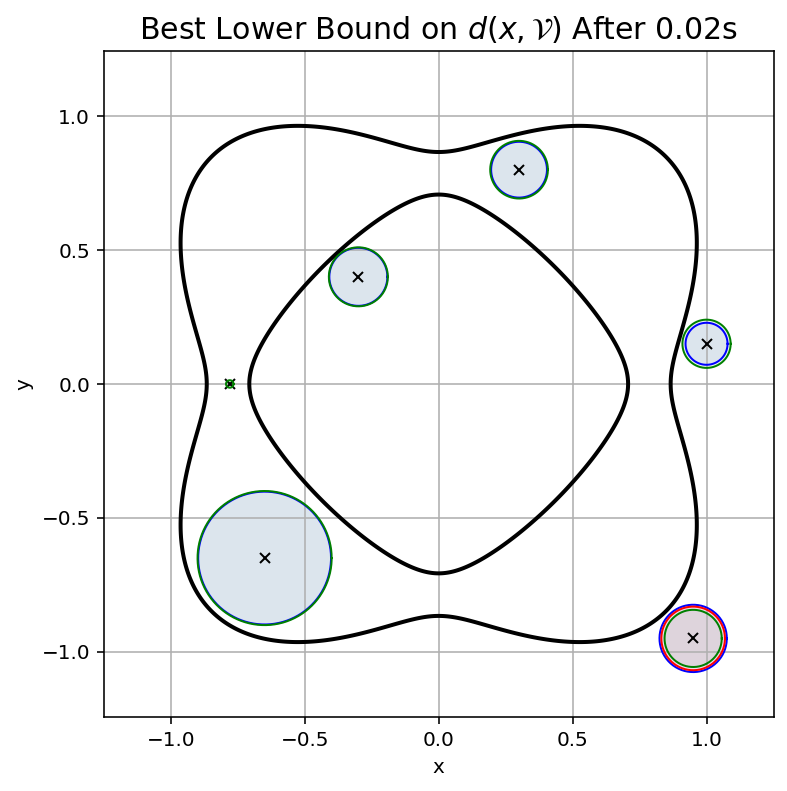}
  \end{subfigure}\hfill
  \begin{subfigure}[b]{0.33\textwidth}
    \centering \includegraphics[width=\linewidth,height=\linewidth]{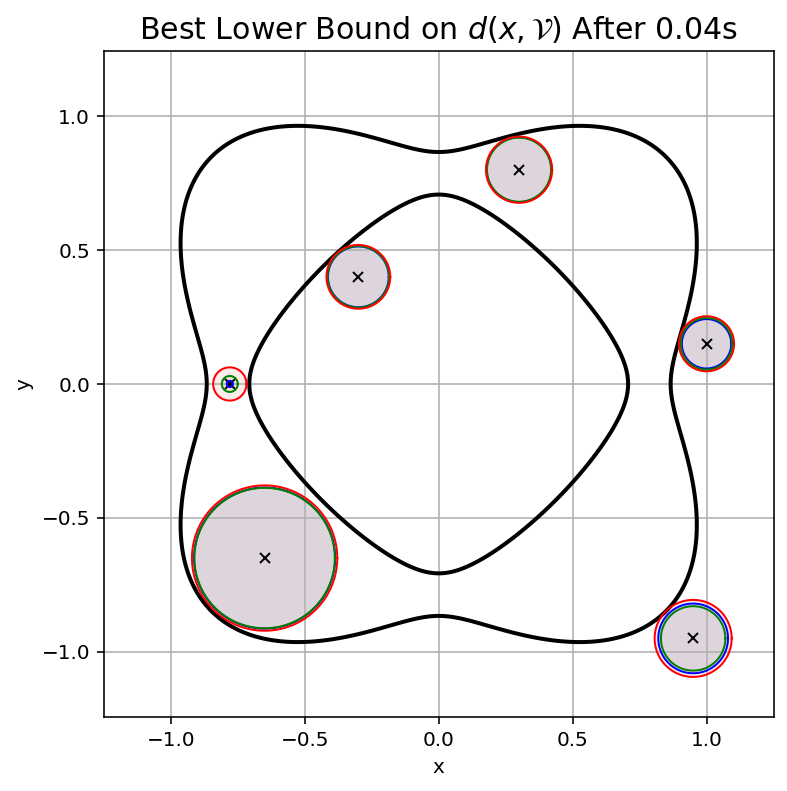}
  \end{subfigure}
  \captionsetup{font=footnotesize} 
  \captionsetup{skip=5pt}
  \caption{Radii representing the best lower bounds on $d(\cdot, \mathcal{V})$ obtained from our spectral hierarchy with initialization Methods 1 (blue) and 2 (green), as well as the corresponding SOS hierarchy (red) after 0.01s, 0.02s, and 0.04s of computation.}
  \label{fig:BDTV}
\end{figure}
%\vspace{-7mm}

To obtain \cref{fig:BDTV}, we precomputed the matrices $\mat{L_{k}}$ and $\mat{M_{k}}(1)$ from \cref{alg: Iterative Construction} -- which depend only on the ideal $\ideal$ -- up to level $k=50$, and then derived spectral hierarchies for each of the displayed points. This highlights an advantage of our construction: when solving several POPs over the same constraint set, a one-time pre-computation step can be leveraged to generate bounds more efficiently.

\FloatBarrier

\subsection{Tensor Spectral Norm}\label{sec:TSN}

Consider the problem of computing the tensor spectral norm of an order-$d$ tensor $\mat{T}\in \R^{n_1\times\dots\times n_d}$, which admits the following formulation:
\begin{equation} \label{eq:TSN}
\norm{\mat{T}}_{\sigma} =\max_{\mat{X}\in\R^{n_1\times\dots\times n_d}} \; \langle \mat{T},\mat{X}\rangle \;\; \mbox{s.t.} \;\; \norm{\mat{X}}_{F}^2=1, \; \mbox{rank}(\mat{X})=1.
\end{equation}
Here $\norm{\mat{X}}_F^2\coloneq\sum_{i_1=1}^{n_1}\dots\sum_{i_d=1}^{n_d} X_{i_1i_2\dots i_d}^2$ denotes the squared Frobenius norm of the tensor $\mat{X}$. The problem \cref{eq:TSN} corresponds to a POP of the form \cref{POP}, where the ideal $\ideal$ is generated by the unit norm constraint $\norm{\mat{X}}_{F}^2=1$ together with all $2\times2$ minors of every matricization of the tensor $\mat{X}$, which vanish simultaneously if and only if $\mat{X}$ has rank one~\cite{Rauhut2020}. This formulation is amenable to our framework of spectral relaxations, with the indeterminates $X_{i_1i_2\dots i_d}$ and the polynomial $1$, appropriately normalized, serving as the set of $\ideal$-spherical polynomials. Moreover,~\cite{Rauhut2020} provides a reduced Gröbner basis for the ideal $\ideal$, which facilitates the computation of a basis for $\R[\mat{X}]/\ideal$ and enables computations mod $\ideal$. %(i.e., rearrangements of its elements into matrices by flattening along its different directions)

We derive bounds on the tensor spectral norm of low-rank, order-3 tensors using the first two levels of our spectral relaxation hierarchy, and compare them against the bounds produced by the first-level SOS, DSOS, and SDSOS relaxations (see \cref{tab:TSN_time,tab:TSN_bounds}). Consistent with our maximum-cut experiments (see \cref{sec: Max-Cut}), we observe that spectral relaxations can be solved for larger problem sizes and require substantially less runtime. The resulting bounds consistently outperform those of the first-level DSOS hierarchy and match exactly those of the first-level SDSOS relaxation. The agreement between our first-level spectral relaxation and the first-level SDSOS relaxation arises because both bounds are given by the Frobenius norm; our second-level spectral relaxation provides a tighter bound.

\setlength{\textfloatsep}{0pt}
\setlength{\tabcolsep}{4pt}
\begin{table}[htbp]
\centering
\small
\begin{tabular}{|l|r|r|r|r|r|}
\hline
\textbf{$n_1\! \times\! n_2\! \times\! n_3$} & SR, $k=1$ & SR, $k=2$ & DSOS & SDSOS & SOS \\
\hline
5×5×5         & 0.0018 & 0.008  & 47.8    & 57.6    & 20.1    \\
10×5×5        & 0.0019 & 0.039  & 817     & 540     & 5664    \\
15×5×5        & 0.0020 & 0.086  & 4608    & X       & X       \\
10×10×5       & 0.0025 & 0.205  & 16146   & X       & X       \\
15×10×5       & 0.0027 & \(\infty\) & \(\infty\) & \(\infty\) & \(\infty\) \\
50×50×50      & 0.218  & \(\infty\) & \(\infty\) & \(\infty\) & \(\infty\)    \\
100×100×100   & 2.11  & \(\infty\) & \(\infty\) & \(\infty\) & \(\infty\)      \\
\hline
\end{tabular}
\captionsetup{font=footnotesize}
\caption{Time (seconds) to obtain bounds on the spectral norm of a random order-3 tensor of size $n_1\! \times\! n_2\! \times\! n_3$ with CP rank $r = \lfloor \min(n_1,n_2,n_3)/2 \rfloor$\tablefootnote{Each tensor is generated as a sum of $r$ rank-one tensors, each formed by outer products of normalized random vectors with uniform $[-1,1]$ entries, scaled by $10/r$ to keep the spectral norm below $10$.} using the 1st and 2nd levels of the spectral relaxation hierarchy with initialization Method 2 (SR), as well as the 1st-level DSOS, SDSOS, and SOS relaxations. An X denotes exceeded memory limits and $\infty$ indicates that the relaxation setup alone exceeded 4 days of runtime and was therefore terminated.}
\label{tab:TSN_time}
\end{table}

\begin{table}[htbp]
\centering
\small
\setlength{\tabcolsep}{3pt}
\begin{tabular}{|l|r|r|r|r|r|}
\hline
\textbf{$n_1\! \times\! n_2\! \times\! n_3$} & SR, $k=1$ & SR, $k=2$ & DSOS & SDSOS & SOS  \\
\hline
5×5×5         & 7.11       & 6.67       & 34.44        & 7.11         & 5.52       \\
10×5×5        & 7.05       & 6.38       & 43.59        & 7.05         & 4.98       \\
15×5×5        & 7.32       & 6.99       & 52.91        & NA           & NA         \\
10×10×5       & 7.39       & 6.98       & 55.05        & NA           & NA         \\

15×10×5        & 6.95   & NA         & NA           & NA           & NA         \\
50×50×50      & 2.00       & NA         & NA           & NA           & NA         \\
100×100×100   & 1.41       & NA         & NA           & NA           & NA         \\

\hline
\end{tabular}
\captionsetup{font=footnotesize}
\caption{Upper bounds on the spectral norm of a random order-3 tensor of size $n_1 \! \times\! n_2 \!\times n_3$ with CP rank $r = \lfloor \min(n_1,n_2,n_3)/2 \rfloor$, obtained from the 1st and 2nd levels of the spectral relaxation hierarchy with initialization Method 2 (SR), as well as the 1st-level DSOS, SDSOS, and SOS relaxations.} 
\label{tab:TSN_bounds}
\end{table} 
%\vspace{-6mm}

The convex hull of the constraint set of \cref{eq:TSN}
coincides with the unit ball of the tensor nuclear norm, which is dual to the spectral norm. From this perspective, the duals of the spectral relaxations we derive for \cref{eq:TSN} correspond to minimizing the linear functional $\langle \mat{T}, \cdot \rangle$ over spectratope outer approximations of the tensor nuclear norm ball (see \cref{sec: Dual Perspective}). This contrasts with the work of~\cite{Rauhut2020}, in which theta body (SDP-representable) outer approximations of the same unit ball are derived and employed for low-rank tensor recovery.

% We also note that an alternative hierarchy of eigencomputations for this problem, based on a different POP formulation for $\norm{\cdot}_{\sigma}$, was already introduced in~\cite{lovitz2024hierarchyeigencomputationspolynomialoptimization}.

%%%%%%%%%%%%%%%%%%%%%%%%%%%%%%%%%%%%%%%%%%%%%%%%%%%%%%%%%%%%%%%%%%%%%%%%%%%%%%%%%%%%%%%%%%%%%%%%%%%%%%%%%%%%%%%%%%%%%%%%%%%%%%%%%%%%%%%%%%%%%%%%%%%%%%%%%%%%%%%%%%%%%%%%%%%%%%%%%%%
\FloatBarrier
%\section{Conclusions \& Future work}
\section{Future directions}

Our work suggests several avenues for future research:
\smallskip

\noindent {\bf Theoretical guarantees:} 
Understanding the convergence of our spectral relaxation hierarchy beyond the case of homogeneous minimization over the sphere remains an open problem. % As discussed in \cref{sec: Characterization}, this task reduces to analyzing our construction for homogeneous POPs over subvarieties of the sphere.

\smallskip
\noindent {\bf Computational considerations:} Our methods rely on access to a nontrivial SOS representation $1\equiv_{\ideal}\sum_{i=1}^m h_i^2$, where $\R[\vec{x}]/\ideal=\mbox{Alg}_\ideal(h_1,\dots, h_m)$. It remains to be understood how one can directly construct such a family of polynomials, given an ideal $\ideal$, and whether one can select them to optimize the computation or the quality of the resulting spectral bounds.  It is also of interest to explore specialized eigensolvers that exploit the structural properties of the matrices arising in our construction.

\smallskip
\noindent {\bf Alternative spectral relaxations:} The matrices underlying our spectral relaxations for the POP defined by $p$ and $\ideal$ depend linearly on the equivalence class $p+\ideal$. Exploring alternative constructions involving nonlinear functions of the coefficients of $p$ could lead to stronger relaxations at a comparable computational cost.

\smallskip

% \noindent {\bf Set approximation questions:} For which ideals $\ideal$ do the spectratope outer approximations from Section \ref{sec: Lambda Bodies} converge —finitely or asymptotically — to $\mbox{conv}(V_\R(\ideal))$? More broadly, how well do spectratopes approximate bounded convex sets under a fixed lifting dimension, and what bounds can be placed on the lifting dimension needed for a given approximation quality? These questions probe the expressive capacity of spectral relaxations in the context of polynomial optimization and beyond.

\noindent {\bf Approximation ratios and rounding:} In many applications, it is of interest to find a good-quality feasible solution to a POP, rather than just a bound on its optimal value. This motivates the development of methods to extract feasible solutions from spectral relaxations with approximation guarantees.

\section*{Acknowledgments} We would like to thank Eitan Levin, Eliza O'Reilly, Ethan Epperly, Mauricio Velasco, and Mitchell Harris for helpful discussions. This work was supported in part by AFOSR grants FA9550-22-1-0225 and FA9550-
23-1-0070.

\bibliographystyle{plain}
\bibliography{references}
\end{document}